\documentclass{article}
\usepackage[utf8]{inputenc}
\usepackage{amsmath}
\usepackage{amssymb}
\usepackage{amsthm}
\usepackage{verbatim}
\usepackage{graphicx}

\usepackage{subcaption}
\usepackage[shortlabels]{enumitem}

\usepackage{xcolor}
\usepackage{float}
\usepackage{lineno}
\usepackage{wasysym}
\usepackage{amsfonts}
\usepackage{tikz}
\usepackage{ulem}

\newtheorem{theorem}{Theorem}[section]

\newtheorem{proposition}[theorem]{Proposition}
\newtheorem{lemma}[theorem]{Lemma}

\newcommand{\AAA}{{\mathsf{A}}}
\newcommand{\AAK}{\mathsf{A}_k}
\newcommand{\BB}{{\mathsf{B}}}
\newcommand{\CC}{{\mathsf{C}}}
\newcommand{\RR}{{\mathsf{R}}}
\newcommand{\QQ}{{\mathsf{Q}}}
\newcommand{\PP}{{\mathsf{P}}}
\newcommand{\II}{{\mathsf{I}}}

\newcommand{\WW}{{\mathsf{W}}}

\newcommand{\JJ}{{\mathcal{J}}}
\newcommand{\YY}{{\mathsf{Y}}}

\newcommand{\wi}{w_{i}}
\newcommand{\wj}{w_{j}}

\newcommand{\Rnn}{{\mathbb{R}^{n\times n}}}
\newcommand{\Rmn}{{\mathbb{R}^{m\times n}}}
\newcommand{\Rpm}{{\mathbb{R}^{p\times m}}}

\newcommand{\xx}{{\mathbf{x}}}
\newcommand{\yy}{{\mathbf{y}}}
\newcommand{\zz}{{\mathbf{z}}}

\newcommand{\OO}{{\mathbf{0}}}
\newcommand{\bb}{{\mathbf{b}}}
\newcommand{\cc}{{\mathbf{c}}}
\newcommand{\uu}{{\mathbf{u}}}
\newcommand{\aaa}{{\mathbf{a}}}

\newcommand{\qq}{{\mathbf{q}}}
\newcommand{\ee}{{\mathbf{e}}}

\newcommand{\rem}[1]{\textcolor{black}{#1}}

\newcommand{\nullspace}[1]{{\mathcal{N}(#1)}}

\DeclareMathOperator*{\argmin}{arg\,min}
\DeclareMathOperator*{\argmax}{arg\,max}

\title{Weighting operators for sparsity regularization}

\author{Ole L{\o}seth Elvetun, Bj{\o}rn Fredrik Nielsen and Niranjana Sudheer}

\begin{document}

\maketitle

\begin{abstract}
Standard regularization methods typically favor solutions which are in, or close to, the orthogonal complement of the null space of the forward operator/matrix $\AAA$. This particular biasedness might not be desirable in applications and can lead to severe challenges when $\AAA$ is non-injective. 

We have therefore, in a series of papers, investigated how to “remedy” this fact, relative to a chosen basis and in a certain mathematical sense: Based on a weighting procedure, it turns out that it is possible to modify both Tikhonov and sparsity regularization such that each member of the chosen basis can be almost perfectly recovered from their image under $\AAA$.  In particular, we have studied this problem for the task of using boundary data to identify the source term in an elliptic PDE. However, this weighting procedure involves $\AAA^\dagger \AAA$, where $\AAA^\dagger$ denotes the pseudo inverse  of $\AAA$, and can thus be CPU-demanding and lead to undesirable error amplification. 

We therefore, in this paper, study alternative weighting approaches and prove that some of the recovery results established for the methodology involving $\AAA^\dagger$ hold for a broader class of weighting schemes. In fact, it turns out that "any" linear operator $\BB$ has an associated proper weighting defined in terms of images under $\BB\AAA$. We also present a series of numerical experiments, employing different choices of $\BB$.   

\end{abstract}

\section{Introduction}

Consider the linear system
\begin{equation}\label{eq:dip}
    \AAA\xx = \yy,
\end{equation}
where $\AAA \in \mathbb{R}^{m \times n}$ has a non-trivial null space. Such problems typically arise in feature selection, signal processing or from the discretization of linear inverse problems. 
Since the matrix $\AAA$ has a null space, it is clear that there does not exist a unique solution to this problem, and a choice has to be made of which kind of solution one seeks. 

For several applications, it makes sense to search for a sparse solution, i.e., \rem{a solution $\xx$ with only a few nonzero components.} 
%where the data is explained by as few components of $\xx$ as possible. 
A popular method to derive such solutions, which has gained much attention in recent decades, is the $\ell^1$-regularization, also known as LASSO \cite{duval2017sparse,Fuchs04,grasmair2011necessary,tibshirani1996regression}. 

The "true" sparsity promoting regularizer would be the cardinality of the support of $\xx$, i.e., the number of nonzero entries in $\xx$. However, this function, which is referred to as the $\|\cdot\|_0$-norm is not convex and results in an NP-hard problem. The $\ell^1$-regularization, however, has shown to be a good proxy in many applications and several important results have been established. To mention a few, recovery of the sources can be guaranteed if the \textit{restricted isometry property} (RIP) \cite{candes05} is satisfied, when the matrix has \textit{low incoherence} \cite{Donoho03}, or if there exists a certain bound on the \textit{exact recovery condition} (ERC) \cite{Tropp04}.

Nevertheless, there are several problems which \rem{neither standard Tikhonov nor standard sparsity regularization} handle very well, e.g., inverse source problems, of which the inverse EEG problem maybe the most well-known. Essentially, the null space of the forward operator causes severe additional challenges: In the limit of a regularized problem, we study %the problem 
\begin{equation*}
    \min_{\xx} \mathcal{R}(\xx) \quad \textnormal{subject to} \quad \AAA\xx = \bb,
\end{equation*}
where $\mathcal{R}$ denotes the regularization functional. \rem{From the first order optimality conditions for the associated Lagrangian, it follows that the optimal solution $\xx^*$ must satisfy 
\begin{equation*}
    \exists \textbf{p} \in \partial \mathcal{R}(\xx^*): \quad p \in \textnormal{Ran}(\AAA^T) = \textnormal{Nul}(\AAA)^\perp, 
\end{equation*}
using the symbol $\partial$ for the subgradient.} 
Consequently, for the most popular choices of $\mathcal{R}$, such as the $\ell^2$- or $\ell^1$-norm, we obtain solutions which are in (or strongly influenced by) the orthogonal complement of the null space of $\AAA$, \rem{cf. Appendix A in \cite{Elv21c} for further details about this issue when standard sparsity regularization is applied to recover the source term in an elliptic PDE}. From a mathematical point of view, this is, for example, what causes the so-called depth bias in the inverse EEG problem \cite{fuchs1999linear,grave1997linear, pascual1999review}.  

\rem{That such biases can occur is well-known, and several suggestions have been made to rectify them \cite{candes2008enhancing,gorodnitsky1995neuromagnetic,Knudsen_2015,lin2006assessing,lucka2012hierarchical,pascual2002standardized,xu07}.} In \cite{Elv21,Elv21c,Elv22,Elv24}  we propose and analyze a weighting scheme defined in terms of the orthogonal projection $\PP=\AAA^\dagger \AAA$ onto the orthogonal complement of the null space of $\AAA$. With this approach, it turns out that a number of almost perfect recovery results can be proven for some classes of source terms. Nevertheless, the method involves the pseudo inverse $\AAA^\dagger$ and is thus CPU-demanding and can lead to severe error amplification. 

These observations motivate the present investigation. That is, we explore alternatives to $\AAA^\dagger$, i.e., weights defined in terms of $\BB \AAA$, where $\BB$ is a linear operator. It turns out that "any" $\BB$ has an associated set of weights for which theorems similar to those presented in \cite{Elv21,Elv21c,Elv22,Elv24} can be established. 
%we attempt to complement this theory with a more general class of weighting operators for which we also can establish theoretical recovery properties. 
This is the main result of the present paper, which is discussed in detail in Section \ref{sec:analysis}. \rem{Section \ref{weighting} contains the definition of the weights and a motivating example.} We close the paper with a series of numerical experiments in Section \ref{sec:numerical_experiments}, illuminating different choices of $\BB$. 

\section{\rem{Weighting and motivation}} %{Problem formulation}
\label{weighting}
Clearly, if $\xx$ solves \eqref{eq:dip}, then it also solves
\begin{equation}\label{eq:mdip}
    \BB\AAA\xx = \BB\yy,
\end{equation}
for any matrix $\BB \in \Rpm$. Now, the matrix $\BB$ can, for example, be the (square root of the) posterior covariance matrix \cite{calvetti2007preconditioned,calvetti2018inverse} - assuming some specific noise, or a specific matrix chosen to enhance some properties in the inverse solution. We will return to this issue in more detail below, but for now, we simply define 
\begin{equation} \label{def:C}
  \CC = \BB\AAA,  
\end{equation}
and consider the variational formulation
\begin{equation}\label{eq:varform}
    \min_\xx \left\{\frac{1}{2}\|\CC\xx - \BB\yy\|_2^2 + \alpha \|\WW\xx\|_1\right\},
\end{equation}
where the diagonal weight matrix $\WW \in \Rnn$ is defined by
\begin{equation}\label{eq:W}
    \WW\ee_i = \wi\ee_i := \|\CC\ee_i\|_2\ee_i \quad \mbox{for } i=1,2,\ldots, n, 
\end{equation}
and $\ee_i$ denotes the standard Euclidean unit basis vector. 
That is, the diagonal entries of $\WW$ are given by 
\begin{equation*}
    \wi = \|\CC\ee_i\|_2, \quad i=1,2, \ldots, n. 
\end{equation*}

For the uniqueness part of some of our results, we need the assumption
\begin{equation}
    \CC \ee_l \neq c \CC \ee_q \quad \mbox{\rem{for all} } l \neq q, c \in \mathbb{R}. \label{eq:nonpar}
\end{equation}
That is, the images under $\CC$ of any two different standard basis vectors must not be parallel. Note that \eqref{eq:nonpar} asserts that none of the basis vectors $\ee_1, \ee_2, \ldots, \ee_n$ belong to the null space of $\CC$. 

The purpose of the present paper is to investigate whether \eqref{eq:varform} can yield more adequate solutions than standard sparsity regularization ($\WW = \II$ and $\BB = \II$). We will present both theoretical and numerical results which illuminate the benefits of the weighting.  

\subsubsection*{Remark} Multiplying with $\BB$ does not change the overall structure of \eqref{eq:dip}. We could therefore have studied a weighted-regularized version of \eqref{eq:dip} instead of \eqref{eq:varform}. Nevertheless, in order to emphasize the role of the choice of $\BB$, we prefer the form \eqref{eq:varform}.   

\subsubsection*{\rem{Motivating example}}
\rem{Let us} consider the task of computing the source term in an elliptic PDE from boundary data: 
\begin{equation}
    \min_{f,u} \| u - d \|^{2}_{L^2( \partial \Omega)}
    \label{eq: problem}
\end{equation}
subject to 
\begin{equation}
    \begin{split}
        -\Delta u + \epsilon u &= f \quad \mbox{in } \Omega, \\
        \frac{\partial u}{\partial \mathbf{n}} &= 0 \quad \mbox{on } \partial \Omega, 
    \end{split}
    \label{eq: bdry_val_prob}
\end{equation}
where $d$ represents Dirichlet boundary data and $\Omega$ denotes the unit square with boundary $\partial \Omega$. Upon discretization, and employing sparsity regularization, we obtain a problem in the form \eqref{eq:varform}, where $\AAA$ is the product of a restriction-to-the-boundary-matrix and the inverse of the matrix associated the differential operator $-\Delta u + u$. Also, $\BB$ is a matrix with suitable dimensions. 

Figure \ref{random_intro} shows the numerical results obtained by solving \eqref{eq:varform} with $\alpha = 10^{-4}$, employing a matrix $\BB$ with random content (drawn from a uniform distribution). A coarse $16 \times 16$ mesh was employed for both forward and inverse computation. The true source is depicted in Figure \ref{true_intro}.  More specifically, $\yy=\AAA \ee_j$ where $j$ is the index associated with the "cell" of the true source, i.e., \eqref{eq:varform} reads, in this special synthetic case, 
\begin{equation} \label{eq:varform_simple_org}
\min_\xx \left\{\frac{1}{2}\|\CC\xx - \CC \ee_j \|_2^2 + \alpha \|\WW\xx\|_1\right\}. 
\end{equation}

We observe that the weighted version successfully recovers the true source, and that the standard approach ($\WW = \II$ and $\BB=\II$) does not produce adequate results, compare panels (b) and (c) in Figure \ref{Intro_eg}.

\begin{figure}[H]
    \centering
    \begin{subfigure}[b]{0.6\linewidth}        %% or \columnwidth
        \centering
        \includegraphics[width=\linewidth]{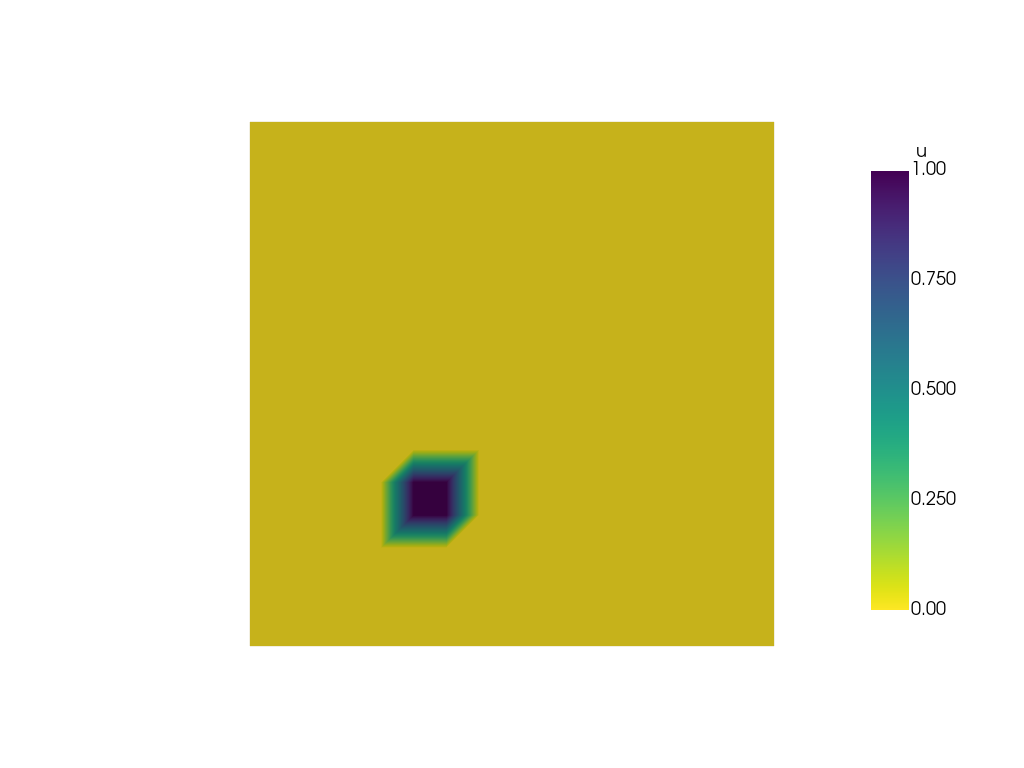}
        \vspace{-3em}
        \caption{True source}
        \label{true_intro}
    \end{subfigure}\par
    \begin{subfigure}[b]{0.6\linewidth}        %% or \columnwidth
        \centering
        \includegraphics[width=\linewidth]{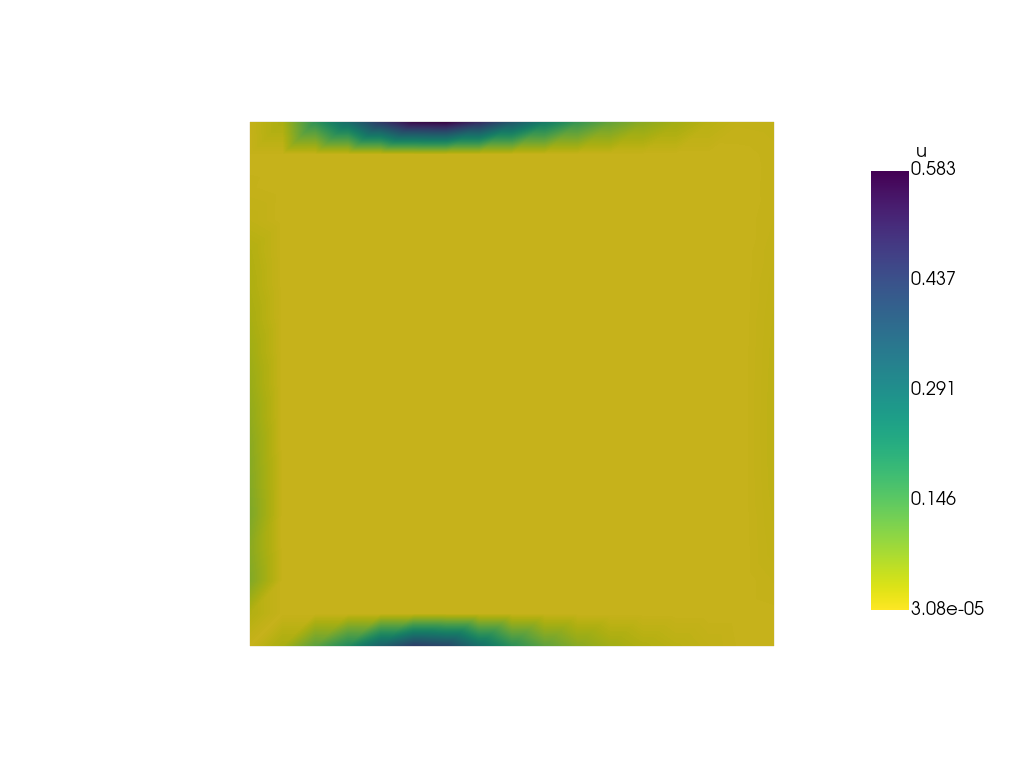}
        \vspace{-3em}
        \caption{Unweighted}
        \label{identity_intro_p}
    \end{subfigure}\par
    \begin{subfigure}[b]{0.6\linewidth}        %% or \columnwidth
        \centering
        \includegraphics[width=\linewidth]{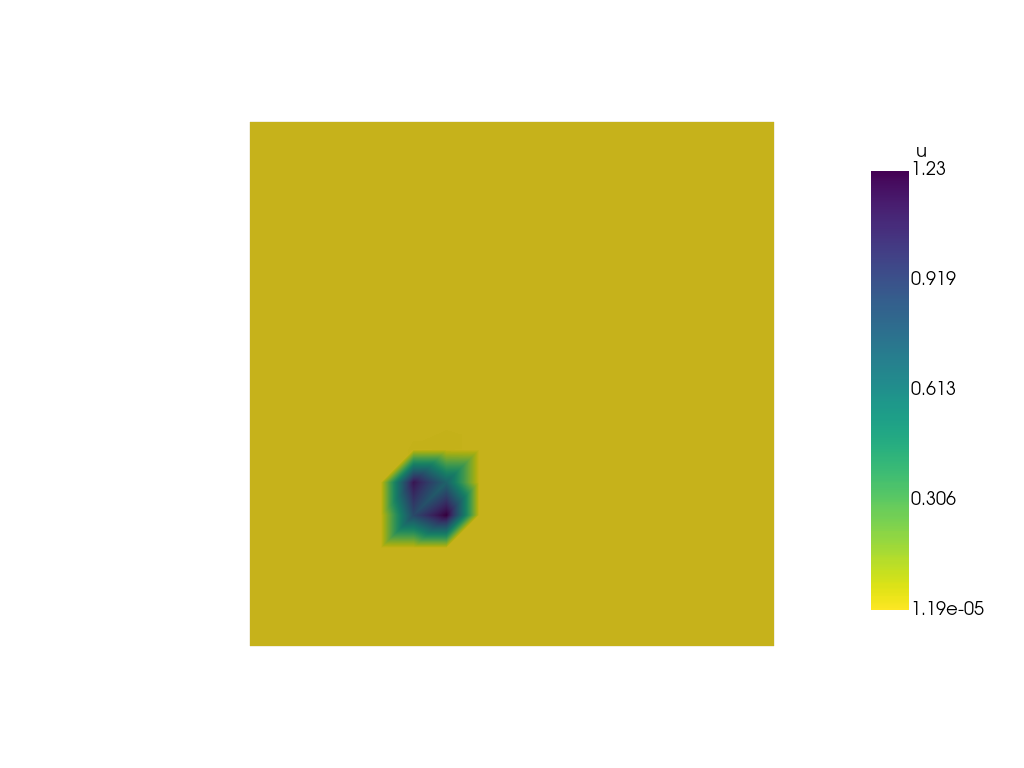}
        \vspace{-3em}
        \caption{Weighted}
        \label{random_intro}
    \end{subfigure}
    \caption{Comparison of standard and weighted sparsity regularization for the screened Poisson problem \eqref{eq: problem} - \eqref{eq: bdry_val_prob}, using $\epsilon = 1$. \rem{Case (b): $\WW=\II$ and $\BB=\II$. Case(c): $\WW$ is as defined in \eqref{eq:W} and $\BB$ has random content.}}
   \label{Intro_eg}
\end{figure}

\section{Analysis}
\label{sec:analysis}
The particular choice $\BB = \AAA^\dagger$, i.e., employing the pseudo inverse of $\AAA$, has been analyzed in a series of papers \cite{Elv21, Elv21c, Elv22}. In this case, $\CC = \AAA^\dagger \AAA = \PP$ becomes the orthogonal projection onto the orthogonal complement of the null space $\nullspace{\AAA}$ of $\AAA$, 
\begin{equation*}
    \PP:\mathbb{R}^n \rightarrow \nullspace{\AAA}^\perp. 
\end{equation*}
With this choice of $\BB$, one can prove that both single and multiple sources can be (approximately) recovered by solving \eqref{eq:varform}, provided that suitable assumptions are fulfilled; \rem{see \cite{Elv21c,Elv24}.} 

However, it might be CPU demanding to compute $\AAA^\dagger$ and employing $\AAA^\dagger$ will typically lead to significant error amplification when $\AAA$ has small positive singular values. One therefore must use an approximation of $\AAA^\dagger$, e.g., the approximation generated by a truncated SVD procedure, $\BB = \AAK^\dagger$, or by invoking Tikhonov regularization. The analysis presented in \cite{Elv21,Elv21c,Elv22} mainly only addresses the case $\BB = \AAA^\dagger$, and not $\BB = \AAA_k^\dagger$, which is rectified by the present paper. 

Furthermore, it turns out that any reasonable matrix $\BB$ yields a weighting that satisfies some basic recovery properties. The proofs of these results are similar to those published in the above mentioned papers, and we thus present them in the appendices, except for two short arguments. Note that Proposition \ref{prop:new_B}, Lemma \ref{lemma:parallel} and Theorem \ref{thm:almost_same_images} have no counterparts in the investigations conducted in our previous work.  

Motivated by the findings presented in Section \ref{weighting}, we will analyze the zero-regularization limit associated with \eqref{eq:varform_simple_org}. More precisely, in the limit $\alpha \rightarrow 0$, the minimization problem \eqref{eq:varform_simple_org} becomes a so-called basis pursuit problem. We now prove the (surprising) fact that "any" $\BB$ used to generate the weights \eqref{eq:W}, see also \eqref{def:C}, will guarantee the recovery of $\ee_j$ from its image $\AAA\ee_j$:
\begin{theorem} \label{thm:basis_pursuit}
    Let $\WW$ be defined as in \eqref{eq:W} and assume that \eqref{eq:nonpar} holds. Then
    \begin{equation}
        \ee_j = \argmin_{\xx} \|\WW\xx\|_1 \quad \textnormal{subject to} \ \ \AAA\xx = \AAA\ee_j. \label{eq:bp}
    \end{equation}
\end{theorem}

\begin{proof}
Let \begin{equation*}
     X_j = \{ \xx \in \mathbb{R}^n: \AAA \xx = \AAA \ee_j \},   
\end{equation*} 
and observe that, if $\xx \in X_j$, then $C\xx = C\ee_j$. Assume that  
\begin{equation*}
    \xx \in X_j , \xx = \sum_{i} c_i\ee_i 
     \quad\textnormal{and}\ \
     \xx \neq \ee_j.  
\end{equation*}  
It follows that
\begin{align*}
  \|\WW\ee_j\|_1 & = w_j\\
   &= \|\CC\ee_j\|_2 \\
   &= \|\CC\xx\|_2 \\
   &= \|\CC(\textstyle\sum_{i}c_i\ee_i)\|_2 \\
   &\leq \textstyle\sum_{i}|c_i|\|\CC\ee_i\|_2\\
   &= \textstyle\sum_{i}w_i|c_i|\\
   &= \|\WW\xx\|_1.
\end{align*}    
If we invoke the assumption \eqref{eq:nonpar}, the triangle inequality above becomes strict and we can therefore conclude that $\xx = \ee_j$ uniquely solves \eqref{eq:bp}. 
\end{proof}

\subsubsection*{\rem{Remark}}
\rem{One may also use the RIP or the mutual incoherence approaches to prove Theorem \ref{thm:basis_pursuit}: Since $\CC = \BB \AAA$, it follows that, if $\ee_j$ solves the problem 
\begin{equation} \label{eq:pre_RIP}
    \min_{\xx} \|\WW\xx\|_1 \quad \textnormal{subject to} \ \ \CC\xx = \CC \ee_j, 
\end{equation}
then $\ee_j$ must also solve the minimization problem in \eqref{eq:bp}. 
With the change of variable $\zz = \WW \xx$ we obtain 
\begin{equation} \label{eq:RIP_form}
    \min_{\xx} \|\zz\|_1 \quad \textnormal{subject to} \ \ \CC \WW^{-1} \zz = \CC \WW^{-1} \qq_j,   
\end{equation}
where $$\qq_i = \WW \ee_i = w_i \ee_i = \| \CC \ee_i \|_2 \ee_i, \quad i=1,2,\ldots, n.$$ 
Now, 
\begin{equation*}
    \| \CC \WW^{-1} \qq_i \|_2 =  \| \CC \ee_i \|_2 = \| \qq_i \|_2, \quad i=1,2,\ldots, n, 
\end{equation*}
and hence the RIP condition is fulfilled, see \cite{candes05}. We can therefore conclude that $\qq_j$ and $\ee_j$ solve \eqref{eq:RIP_form} and \eqref{eq:pre_RIP}, respectively. Furthermore, 
\begin{equation*}
    |(\CC \WW^{-1} \ee_k, \CC \WW^{-1} \ee_l)| = \left| \left( \frac{\CC \ee_k}{\| \CC \ee_k \|_2}, \frac{\CC \ee_l}{\| \CC \ee_l \|_2} \right) \right| < 1, \quad k \neq l, 
\end{equation*}
provided that \eqref{eq:nonpar} holds. Hence, $\CC \WW^{-1}$ satisfies the incoherence condition \cite{Donoho03,Tropp04} and it follows that $\qq_j$ is the only solution to \eqref{eq:RIP_form}, keeping in mind that $\qq_j$ equals $\ee_j$ times a scalar. 
}

The computational results reported above in Section \ref{weighting} are not surprising in view of Theorem \ref{thm:basis_pursuit}, even though we employed a matrix $\BB$ with random content to define $\CC$, see \eqref{def:C} and \eqref{eq:W}. 

In order to analyze the regularized problem \eqref{eq:varform_simple_org}, we need a result concerning the mathematical properties of $\WW^{-1}\CC^T\CC\ee_j$. That is, we will prove that the $j$'th component of $\WW^{-1}\CC^T\CC\ee_j$ is the largest component of this vector. One might consider this to be a generalization of Theorem 4.2 in \cite{Elv21}, which proves this result for the special case $\CC = \AAA^\dagger \AAA$, i.e., $\BB=\AAA^\dagger$, cf. \eqref{eq:mdip} and \eqref{def:C}. We will use the following lemma at several occasions below.  
\begin{lemma}
\label{lemma}
    Let $\WW$ be defined as in \eqref{eq:W} and assume that \eqref{eq:nonpar} holds. Then
    \begin{equation} \label{eq:max_property}
        j = \argmax_{i} |(\WW^{-1}\CC^T\CC\ee_j,\ee_i)|.
    \end{equation}
\end{lemma} 
Note the following before we prove this lemma: Considering the equation 
\begin{equation} \label{eq:very_simple_equation}
\CC \xx =\CC \ee_j,
\end{equation} 
we observe that \eqref{eq:max_property} shows that the index $j$ of the "true" source $\ee_j$ can be identified from its 
image $\CC \ee_j$ by employing the inverse of the weight matrix $\WW$. When $\BB = \AAA^\dagger$, 
$\WW^{-1}\CC^T\CC\ee_j = \WW^{-1}\AAA^\dagger\AAA\ee_j$, because $\CC = \AAA^\dagger \AAA$ is a projection, which can be interpreted as a re-weighted version of the minimum norm solution $\AAA^\dagger\AAA\ee_j$ of $\AAA \xx = \AAA \ee_j$.  The proof of Lemma \ref{lemma} is short: 
\begin{proof}
    Recall the definition \eqref{eq:W} of the diagonal weight matrix $\WW$. We have
    \begin{align}
        \nonumber
        (\WW^{-1}\CC^T\CC\ee_j, \ee_i) &= \left(\CC^T\CC\ee_j,\frac{\ee_i}{\wi}\right) \\ &=
        \nonumber
        \left(\CC\ee_j,\frac{C\ee_i}{\wi}\right) \\ &= 
        \label{eq:maximum_expression}
        \|\CC\ee_j\|_2 \left(\frac{\CC\ee_j}{\|\CC\ee_j\|_2},\frac{\CC\ee_i}{\|\CC\ee_i\|_2}\right) \mbox{ for } i=1,2,\ldots,n.
    \end{align}
    The result now follows from the Cauchy-Schwartz inequality and assumption \eqref{eq:nonpar}.
\end{proof}

The result concerning the minimization problem \eqref{eq:varform_simple_org}, studied in Section \ref{weighting}, reads as follows: 
\begin{theorem}\label{thm:varThm}
   Let $\WW$ be defined as in \eqref{eq:W}. %and assume that \eqref{eq:nonpar} holds. 
   Then $\xx_\alpha = \gamma_\alpha \ee_j$ is a minimizer of 
    \begin{equation} \label{eq:varform_simple}
      \min_\xx \left\{\frac{1}{2}\|\CC\xx - \CC\ee_j\|^2 + \alpha \|\WW\xx\|_1\right\},
    \end{equation}
    where $\gamma_\alpha = 1 - \frac{\alpha}{\wj}$. If \eqref{eq:nonpar} holds, then $\xx_\alpha = \gamma_\alpha \ee_j$ is the unique solution of \eqref{eq:varform_simple}. 
\end{theorem}
\begin{proof}
    The proof of this result is similar to the argument for Theorem 4.3 in \cite{Elv21c}. See Appendix \ref{proof_thm:varThm} for further details. 
\end{proof}
Theorem \ref{thm:varThm} asserts that the support of the true source $\ee_j$ is preserved by the solution of \eqref{eq:varform_simple}, cf. the numerical results presented in Section \ref{weighting}. Furthermore, the solution $\xx_\alpha$ of \eqref{eq:varform_simple} converges toward $\ee_j$ as $\alpha \rightarrow 0$. 

The possibility of identifying several sources and sinks can be analyzed in terms of the existence of a Lagrange multiplier $\cc$, also referred to as a dual certificate; see \cite{duval2017sparse,Fuchs04,grasmair10}. %\cite{bernstein2020sparse,duval2017sparse,Fuchs04,grasmair10}. 
The proof of the following theorem is omitted because it is a straightforward generalization of Theorem 4.1 in \cite{Elv24}. 
\begin{theorem}
\label{thm:support} 
 Let $\xx^* = \textstyle\sum_{\mathcal{J}}x_{j}^*\ee_{j}$. Assume that there exists a vector $\cc$ which satisfies the following conditions 
  \begin{align}
  \label{cond_1}
\frac{\CC\ee_i}{\|\CC\ee_i\|}\cdot \cc &= sgn(x_{i}^{*}), \quad \forall i \in \mathcal{J}, \\  
      \label{cond_2}
          \left |\frac{\CC\ee_i}{\|\CC\ee_i\|}\cdot \cc \right| &< 1, \quad  \forall i \in \mathcal{J}^c,                
  \end{align}
where $\mathcal{J} = supp(\xx^*)$ and $\mathcal{J^{\cc}} = \{ 1,2, \dots , n \} \backslash \mathcal{J}. $ Then $\xx^*$ solves the basis pursuit problem  
\begin{equation} \label{eq:basis_pursuit_several_sources_sinks}
  \min_\xx \|\WW \xx\|_1\quad \text{subject to} \quad \AAA \xx = \AAA \xx^*. 
\end{equation}
Furthermore, if $\yy$ is any other solution of \eqref{eq:basis_pursuit_several_sources_sinks}, then  
 \begin{equation*}
     supp(\yy) \subseteq supp(\xx^{*}). 
 \end{equation*}
\end{theorem}
\begin{proof}
    Omitted, cf. the argument for Theorem 4.1 in \cite{Elv24}. %See Appendix \ref{proof_thm:support}. 
\end{proof}

The existence of a dual certificate $\cc$ satisfying \eqref{cond_1}-\eqref{cond_2} can be guaranteed under beneficial circumstances. For example, \rem{when a certain disjoint property holds}: 
\begin{theorem} \label{thm:disjoint_supports}
  \rem{Let $\mathcal{J} = \text{supp}(\xx^*)$ and assume that 
  \begin{equation} \label{eq:disjoint_assumption} 
     \text{supp}(\CC^T\CC \ee_j) \cap \text{supp}(\CC^T\CC \ee_k) = \emptyset \quad \text{for all} \ j, k \in \mathcal{J}, \ j \neq k.
    \end{equation}
    }
  Then $ \xx^* = \sum_{j \in \mathcal{J}} x_j^* \ee_j $ is the unique solution to the problem
 \begin{equation}
  \min_{\xx \in \mathbb{R}^n} \|\WW \xx\|_1 \quad \text{subject to} \quad \AAA \xx = \AAA \xx^*,
  \label{II-Ax =Ax*}
  \end{equation}
  provided that \eqref{eq:nonpar} holds.
\end{theorem}
\begin{proof}
    \rem{We first show that \eqref{eq:disjoint_assumption} implies that 
    \begin{equation} \label{eq:orthogonality_assumption}
        (\CC \ee_j, \CC \ee_k) = 0  \quad \text{for all} \ j, k \in \mathcal{J}, \ j \neq k.
    \end{equation}
    For any $j \in \mathcal{J}$, assumption \eqref{eq:nonpar} yields that 
    \begin{equation*}
        (\CC^T \CC \ee_j, \ee_j) = \| C \ee_j \|_2^2 \neq 0 \Rightarrow j \in \text{supp}(\CC^T\CC \ee_j). 
    \end{equation*}
    Let $j,k \in \mathcal{J}$ be arbitrary. Then, $j \in \text{supp}(\CC^T\CC \ee_j)$ and it follows from \eqref{eq:disjoint_assumption} that 
    \begin{equation*}
        k \notin \text{supp}(\CC^T\CC \ee_j) \Rightarrow (\CC^T\CC \ee_j, \ee_k) = 0 \Rightarrow (\CC \ee_j, \CC \ee_k) = 0.
    \end{equation*}
    %and we conclude that the vectors $\{\CC\ee_j\}_{j\in \mathcal{J}}$ are orthogonal.
    }
    
    \rem{The rest of the proof, including the use of \eqref{eq:orthogonality_assumption}, is rather similar to the argument for Theorem 4.2 in \cite{Elv24}, see Appendix \ref{proof_thm:disjoint_supports} for further details}.  
\end{proof}

\rem{Provided that $\{\AAA\ee_j\}_{j\in \mathcal{J}}$ is a linearly independent set, we will now briefly explain that one can always construct a matrix $\BB$ such that $C=\BB \AAA$ satisfies the orthogonality property \eqref{eq:orthogonality_assumption} needed in the proof of Theorem \ref{thm:disjoint_supports}. Note that we have not succeeded in designing $\BB$ such that the disjoint support assumption \eqref{eq:disjoint_assumption} holds, only that its consequence \eqref{eq:orthogonality_assumption} is fulfilled.}  

\rem{Let us introduce the notation}
\begin{align}
    \nonumber
    \AAA &= [\aaa_1 \; \aaa_2 \; \ldots \; \aaa_n] \in \mathbb{R}^{m \times n}, \\
    \nonumber
    \mathcal{J} &= \{ j_1, \, j_2, \ldots, j_s \},  \\
    \label{def_Y}
    \YY &= [\aaa_{j_1} \; \aaa_{j_2} \; \ldots \; \aaa_{j_s}] \in \mathbb{R}^{m \times s}. 
\end{align}

\begin{proposition} \label{prop:new_B}
    Assume that $\{\AAA\ee_j\}_{j\in \mathcal{J}} = \{ \aaa_{j_1}, \, \aaa_{j_2}, \ldots, \, \aaa_{j_s}\}$ is a linearly independent set and that $s \leq m \leq n$. Then $\CC = \YY^\dagger \AAA$ satisfies \rem{\eqref{eq:orthogonality_assumption}}, where $\YY$ is defined in \eqref{def_Y}. 
\end{proposition}
\begin{proof}
    Since $\{ \aaa_{j_1}, \, \aaa_{j_2}, \ldots, \, \aaa_{j_s}\}$ are linearly independent, the null space $\nullspace{\YY}$ of $\YY$ only contains the zero element. Therefore, $\YY^\dagger \YY$ equals the identity because it yields the orthogonal projection onto the orthogonal complement $\nullspace{\YY}^\perp$ of $\nullspace{\YY}$, i.e., 
    \begin{equation*}
        \YY^\dagger \YY \hat{\ee}_k = \hat{\ee}_k \quad \mbox{for } k=1,2, \ldots, s,  
    \end{equation*}
    where $\hat{\ee}_{j_k} \in \mathbb{R}^s$ denotes the standard unit basis vector containing only zero components, except for the $k$'th component which equals $1$. From the definition of $\YY$ we find that 
    \begin{equation*}
        \AAA \ee_{j_k} = \aaa_{j_k} = \YY \hat{\ee}_k \quad \mbox{for } k=1,2, \ldots, s,  
    \end{equation*}
    and hence 
    \begin{equation} \label{eq:simple_forward_images}
        \YY^\dagger \AAA \ee_{j_k} = \YY^\dagger \YY \hat{\ee}_k = \hat{\ee}_k \quad \mbox{for } k=1,2, \ldots, s. 
    \end{equation}
    This shows that $\{ \YY^\dagger \AAA \ee_{j_k} \}_{k=1}^s$ is a set of orthogonal vectors.  
\end{proof}
We will use $B=\YY^\dagger$ in some of the numerical experiments presented below and discuss why this approach might be beneficial for the task of identifying several sources and sinks, using boundary data, for the model problem studied in Section \ref{weighting}. We also note that condition \rem{\eqref{eq:disjoint_assumption}} in Theorem \ref{thm:disjoint_supports}, due to \eqref{eq:simple_forward_images}, now can be written in the form $$\text{supp}(\CC^T \hat{\ee}_j) \cap \text{supp}(\CC^T \hat{\ee}_k) = \emptyset \quad \text{for all} \ j, k \in \mathcal{J}, \ j \neq k,$$ when $C=\YY^\dagger \AAA$. It is thus sufficient to check whether the $j$'th and $k$'th rows of $\CC$ have disjoint supports for all $j, k \in \mathcal{J}, \ j \neq k$, which is easy to do with a computer. 

Compared with Theorem \ref{thm:disjoint_supports}, the next result concerns the near diametrically opposite case, namely when the images $\{\CC\ee_j\}_\JJ$ of a collection of sources $\{\ee_j\}_\JJ$ are almost parallel. Recall that Theorem \ref{thm:support} provides two sufficient conditions for the recovery of $\xx^* = \textstyle\sum_{\mathcal{J}}x_{j}^*\ee_{j}$.  
These conditions rely on the existence of a dual vector $\cc$. If $x_{j}^* > 0 \ \forall j \in \JJ$ and we have the idealized case that all pairwise inner products of their normalized images under $\CC$ are equal to some constant $\hat{\rho} \in (0,1)$, i.e., 
\begin{equation} \label{eq:parallel_images}
\left(\frac{\CC\ee_j}{\|\CC\ee_j\|},\frac{\CC\ee_{j'}}{\CC\ee_{j'}\|}\right) = \hat{\rho}, \quad j, j' \in \JJ, j \neq j',
\end{equation} 
we will see that condition \eqref{cond_1} holds if there exists a vector $\zz \geq \mathbf{0}$ such that $\QQ\zz = \mathbf{1}$, where the matrix $\QQ = \QQ(\hat{\rho})$ is specified in Lemma \ref{lemma:parallel} below. With some mild additional conditions, it also turns out that \eqref{cond_2} will be satisfied. Consequently, Theorem \ref{thm:support} can be applied. 

If the inner products of the normalized images under $\CC$ are not exactly identical, i.e., \eqref{eq:parallel_images} is only approximately satisfied, then the analysis becomes more technical: We can use relatively standard perturbation theory for matrices to obtain bounds on how much deviation from $\hat{\rho}$ in \eqref{eq:parallel_images} which is tolerable and still prove that conditions \eqref{cond_1} and \eqref{cond_2} hold. We present these considerations in the form of a lemma and a theorem.

\begin{lemma}\label{lemma:parallel}
    Let $\QQ = \QQ(\rho) \in \mathbb{R}^{s \times s}$, $s>1$, be symmetric with entries  
    \begin{equation*}
        q_{ij} = \begin{cases}
            1, &i = j, \\ 
            \rho, &i \neq j,
        \end{cases}
    \end{equation*}
    where $0 < \rho < 1$. Then,
    \begin{enumerate}[(i)]
      \item the vector $\yy$ with all components equal to $$y_i = \frac{1}{1+(s-1)\rho}$$ solves $\QQ\xx = \mathbf{1}$ uniquely.
      \item if a matrix $\RR \in \mathbb{R}^{s \times s}$ obeys the bound
      \begin{equation*}
        \|\RR\|_\infty \leq \frac{(1-\rho)(\rho(s-1)+1)}{2\rho(2s-3)+2},
      \end{equation*}
      then the unique solution $\bar{\xx}$ of $$(\QQ + \RR)\xx = \mathbf{1}$$ only has non-negative components.
    \end{enumerate}
\begin{proof} \mbox{ } 
    \begin{enumerate}[(i)] 
    \item  Notice that each row of the matrix $\QQ$ sums to $1 + (s-1)\rho$. We can   
        therefore conclude that the vector $\yy$ given by
        \begin{equation}\label{eq:yi}
                y_i = \frac{1}{1+(s-1)\rho} > 0, \quad i \in \{1, \ldots, s\}, 
        \end{equation}
        solves $\QQ\xx = \mathbf{1}.$ 

        We next derive that the inverse $\QQ^{-1}$ of $\QQ$ exists and has the same structure as $\QQ$. This can be verified in a straightforward manner by solving the $2 \times 2$ system, derived from the condition that $\QQ^{-1} \QQ = \II$, 
        \begin{equation} \label{eq:2x2system}
            \begin{cases}
                \ d + (s-1)\rho\zeta &= 1, \\
                \rho d + (s-2)\rho\zeta + \zeta &= 0,
            \end{cases}
        \end{equation}
        where $d$ and $\zeta$ are the diagonal and off-diagonal entries of $\QQ^{-1}$, respectively. This shows that $\yy$ is the unique solution of $\QQ\xx = \mathbf{1}$.
        
    \item From \cite[Theorem 8.9]{lyche00}, we have that, if $r := \|\QQ^{-1}\|_\infty\|\RR\|_\infty < 1$, then there exists a unique solution $\bar{\xx}$ of 
    $$(\QQ + \RR)\xx = \mathbf{1},$$
    obeying the bound
    \begin{equation} \label{eq:Lyches_book}
        \frac{\|\yy - \bar{\xx}\|_\infty}{\|\yy\|_\infty} \leq \frac{r}{1 - r}.
    \end{equation}
    Recall the expression \eqref{eq:yi} for the components of $\yy$, which are all identical and positive. If one of the components of $\bar{\xx}$ is negative, say $\bar{x}_i < 0$, then 
    \begin{align*}
        \frac{\|\yy - \bar{\xx}\|_\infty}{\|\yy\|_\infty} &\geq \frac{|y_i - \bar{x}_i|}{|y_i|} \\
        & > \frac{|y_i|}{|y_i|} \\
        &= 1.
    \end{align*}
    Hence, from inequality \eqref{eq:Lyches_book} we can conclude that all the components of $\bar{\xx}$ are non-negative provided that
    \begin{equation*}
         \frac{r}{1 - r} = \frac{\|\QQ^{-1}\|_\infty\|\RR\|_\infty}{1 - \|\QQ^{-1}\|_\infty\|\RR\|_\infty} \leq 1,
    \end{equation*}
    which holds whenever 
    \begin{equation}\label{eq:RRbound}
        \|\RR\|_\infty \leq \frac{1}{2\|\QQ^{-1}\|_\infty}.
    \end{equation}

    By solving the $2 \times 2$ system \eqref{eq:2x2system} for $d$ and $\zeta$, and observing that summing the absolute values of the entries of any row of $\QQ^{-1}$ gives the same number, we compute the matrix norm
    \begin{equation*}
        \|\QQ^{-1}\|_\infty = \frac{\rho(2s-3) + 1}{(1-\rho)(\rho(s-1)+1)}.
    \end{equation*}
    Consequently, by inserting this into \eqref{eq:RRbound} we get the bound
    \begin{equation*}
        \|\RR\|_\infty \leq \frac{(1-\rho)(\rho(s-1)+1)}{2\rho(2s-3)+2}.
    \end{equation*}
    This completes the proof.
    \end{enumerate}
\end{proof}
\end{lemma}

We mentioned before Lemma \ref{lemma:parallel} that our strategy will be to find a dual vector $\cc$ such that \eqref{cond_1} and \eqref{cond_2} are satisfied. Inspired by the choice of $\cc$ in the proof of Theorem \ref{thm:disjoint_supports}, see Appendix \ref{proof_thm:disjoint_supports}, we will employ a dual certificate in the form  
\begin{equation*}
    \cc = \sum_{\JJ} z_j \frac{\CC\ee_j}{\|\CC\ee_j\|}.
\end{equation*}
Here, $\zz = [z_1 \ z_2 \ \ldots \ z_s]^T$ is determined by solving the linear system, cf. \eqref{cond_1}, 
\begin{equation*}
    \left(\frac{\CC\ee_i}{\|\CC\ee_i\|}, \cc\right) = 1 \quad \forall i \in \JJ
\end{equation*}
or 
\begin{equation} \label{eq:lin_syst_dual_certificate}
    \sum_{\JJ} z_j \left(\frac{\CC\ee_i}{\|\CC\ee_i\|}, \frac{\CC\ee_j}{\|\CC\ee_j\|}  \right) = 1 \quad \forall i \in \JJ,
\end{equation}
assuming that the true source $\xx^*=\sum_{\JJ} x_j^*\ee_j$ only has positive components. Note that when \eqref{eq:parallel_images} holds, \eqref{eq:lin_syst_dual_certificate} becomes the linear system $\QQ \xx = \mathbf{1}$, with $\rho=\hat{\rho}$, studied in Lemma \ref{lemma:parallel}(i), whereas for the case when \eqref{eq:parallel_images} only is approximately satisfied, we get the problem $(\QQ + \RR)\xx = \mathbf{1}$ explored in Lemma \ref{lemma:parallel}(ii). Furthermore, the main diagonals of $\QQ$ and $\QQ+\RR$ only contain ones, e.g., the main diagonal of $\RR$ consists of zeros. The details are presented in the following theorem.
\begin{theorem} \label{thm:almost_same_images}
Let $\xx^* = \textstyle\sum_{\mathcal{J}}x_{j}^*\ee_{j}$, where we assume that $x_j^*$ has the same sign for all $j \in \JJ$. We introduce the notation
\begin{equation} \label{def:g_ij}
g_{ij} = \left(\frac{\CC\ee_i}{\|\CC\ee_i\|},\frac{\CC\ee_{j}}{\|\CC\ee_{j}\|}\right),
\end{equation} 
and assume that
\begin{equation} \label{eq:most_parallel}
    g_{jj'} > |g_{ij}|, \quad \forall j, j' \in \JJ \textnormal{ and } \forall i \in \JJ^c.
\end{equation}
Furthermore, define the $s \times s$ matrix $\RR(\rho)$ by 
\begin{equation} \label{def:R_rho}
\RR(\rho)_{kl} = 
    \begin{cases} 0, &k = l, \\
    g_{j_kj_l} - \rho, &k \neq l,
    \end{cases}
\end{equation}
\rem{for $k,l \in \{ 1,2,\ldots,s \}$, i.e., $j_k,j_l \in \JJ$.}
If there exists $\bar{\rho} \in (0,1)$ such that 
\begin{equation} \label{eq:upper_bound_assumption}
    \|\RR(\bar{\rho})\|_\infty \leq \frac{(1-\bar{\rho})(\bar{\rho}(s-1)+1)}{2\bar{\rho}(2s-3)+2}
\end{equation} 
then $\xx^*$ is a solution of the basis pursuit problem
\begin{equation*}
  \min_\xx \|\WW \xx\|_1\quad \text{subject to} \quad \AAA \xx = \AAA \xx^*. 
\end{equation*}
\end{theorem}

Before we prove this result, we remark that: Roughly speaking, condition \eqref{eq:most_parallel} asserts that the members of $\{ \CC \ee_j/\| \CC \ee_j \| \}_\JJ$ are "more parallel/aligned" with other members of this set than with the members of  $\{ \CC \ee_j/\| \CC \ee_j \| \}_{\JJ^c}$. Also, if \eqref{eq:parallel_images} holds, then \eqref{eq:upper_bound_assumption} is satisfied with $\bar{\rho}=\hat{\rho}$ because in this case $R(\hat{\rho})$ becomes the zero matrix, i.e., a matrix only containing zeros. When \eqref{eq:parallel_images} only is approximately fulfilled, we can still prove that $\xx^*$ solves the basis pursuit problem, provided that \eqref{eq:upper_bound_assumption} holds. In this case $\bar{\rho} \in (0,1)$ is a suitable number such that 
\begin{equation*} 
\left(\frac{\CC\ee_j}{\|\CC\ee_j\|},\frac{\CC\ee_{j'}}{\CC\ee_{j'}\|}\right) \approx \bar{\rho}, \quad j, j' \in \JJ, j \neq j',
\end{equation*} 

\begin{proof}
    We assume that $x_j^* > 0, \forall j \in \JJ$. The proof is analogous for the negative case. Define the vector $\cc$ by
    \begin{equation*}
        \cc = \sum_{j \in \JJ} z_j \frac{\CC\ee_j}{\|\CC\ee_j\|},
    \end{equation*}
    and let $\QQ(\rho) \in \mathbb{R}^{s \times s}$ be as defined in Lemma \ref{lemma:parallel}. Note that the condition \eqref{cond_1} in the current setup reads
    \begin{equation} \label{eq:modified_cond_1}
        (\QQ(\bar{\rho}) + \RR(\bar{\rho}))\zz = \mathbf{1},
    \end{equation}
    cf. the definition of $\QQ(\rho)$ in Lemma \ref{lemma:parallel}, the definition \eqref{def:R_rho} of $\RR(\rho)$ and the definition \eqref{def:g_ij} of $g_{ij}$. 
    From Lemma \ref{lemma:parallel} and the upper bound assumption \eqref{eq:upper_bound_assumption} on $\|\RR(\bar{\rho})\|_\infty$, we have that there exists a unique solution $\zz$ to \eqref{eq:modified_cond_1} for which all components are non-negative, i.e, $z_j \geq 0$ for all $j \in \JJ$. 

    To conclude, we observe that also \eqref{cond_2} is satisfied since, for any $i \in \JJ^c$ and for any $j' \in \JJ$,
    \begin{equation*}
        \sum_{j\in\JJ} \left|\left(\frac{\CC\ee_i}{\|\CC\ee_i\|}, z_j \frac{\CC\ee_j}{\|\CC\ee_j\|}\right)\right| 
        < \sum_{j\in\JJ} z_j\left(\frac{\CC\ee_{j'}}{\|\CC\ee_{j'}\|},  \frac{\CC\ee_j}{\|\CC\ee_j\|}\right) = 1,
    \end{equation*}
    where we used the assumption that $g_{jj'} > |g_{ij}|$ and the fact that $z_j \geq 0$ to obtain the inequality. The equality follows from \eqref{eq:modified_cond_1} because the entries in row number $j'$ of $(\QQ(\bar{\rho})+\RR(\bar{\rho}))$ have the form $\left(\frac{\CC\ee_{j'}}{\|\CC\ee_{j'}\|},  \frac{\CC\ee_j}{\|\CC\ee_j\|}\right)$, see \eqref{def:R_rho}, \eqref{def:g_ij} and the definition of $\QQ(\rho)$ in Lemma \ref{lemma:parallel}.
    %observing that the sum if equal to component $j'$ in the vector $(\QQ(\bar{\rho})+\RR(\bar{\rho}))\zz$, which is equal to one.  
\end{proof}

\section{Numerical experiments}
\label{sec:numerical_experiments}
In this section, we visualize the effects of applying different choices of $\BB$ and discuss some computational results in view of the analysis presented in Section \ref{sec:analysis}. 
We performed all the simulations on a uniform grid and employed the finite element method (FEM) to discretize the elliptic operator involved in the boundary value problem \eqref{eq: bdry_val_prob}, using first-order Lagrange elements. Note that the discretization of the forward operator $f \mapsto u|_{\partial\Omega}$ yields a (forward) matrix $\AAA \in \Rmn$ in the form
\begin{equation*}
    \mathbf{\AAA} = \mathbf{M}_{\partial}^{1/2} \mathbf{L}^{-1} \mathbf{M},
\end{equation*}
where $\mathbf{L}$ is the matrix associated with $-\Delta u + \epsilon u$, and $\mathbf{M}$ and $\mathbf{M}_\partial$ are the standard and boundary mass matrices, respectively, see, e.g., \cite{Elv21c} for further details. \normalem
\emph{It is important to keep in mind that we try to recover internal sources using boundary data only. That is, $\AAA$ has a large null space.} 

All element matrices were generated using the FEniCSx software \cite{logg2012automated}, and no noise was added to the generated (synthetic) data, except for the data used to produce the results presented in Figure \ref{helm}.  Inverse crimes were avoided by generating the forward data on a $128 \times 128$ grid, whereas a $64 \times 64$ mesh was used to solve the inverse problems. However, to be in perfect alignment with the theory, an exception was made for the simulations displayed in Figure \ref{coarser_multi}: A coarse grid of size $16 \times 16$ was applied for both the forward and inverse computations. \\

All figures below display the solution to the optimization problem
\begin{equation*}
    \min_{\xx } \left \{\frac{1}{2}\|\CC\xx - \BB \yy \|_2^2 + \alpha \|\WW\xx\|_1 \right \}, 
\end{equation*}
where the matrices $\CC = \BB\AAA$ and $\WW$ are defined in Section \ref{weighting}, and $\BB$ is either 
 \begin{itemize}
    \item $\II$ - the identity matrix,
    \item $\AAK^\dagger$ - the truncated pseudo-inverse of $\AAA$ employing $k = 100$ singular values for the noise free cases and $k = 10$ for the case with added noise,    
    \item $\BB_{r}$ - a random sparse matrix, or 
    \item $\YY^{\dagger}$ - submatrix of $\AAA$ formed by selecting certain subcolumns. 
\end{itemize}
If not stated otherwise, the regularization parameter was $\alpha = 10^{-4}$.

\subsection{Single and multiple composite sources}
Theorem \ref{thm:varThm} guarantees that a source represented by a single basis vector can be recovered, albeit with a slightly smaller magnitude. In this first example, we deviate somewhat from this scenario and rather consider a source represented by several basis vectors which are spatial neighbours, i.e., a composite source, cf. panel (a) in Figure \ref{one_source}. Panels (b) - (d) show the inverse solutions computed with different choices of $\BB$. We observe that, for all the choices of $\BB$, the support of the inverse solutions are located inside the support of the true source. The sparsity-promoting feature of the $\ell^{1}$-norm might explain the smaller support and the larger magnitude of the inverse solutions, compared with the true source.

   \begin{figure}[H]
    \centering
    \begin{subfigure}[b]{0.45\linewidth}        %% or \columnwidth
        \centering
        \includegraphics[width=\linewidth]{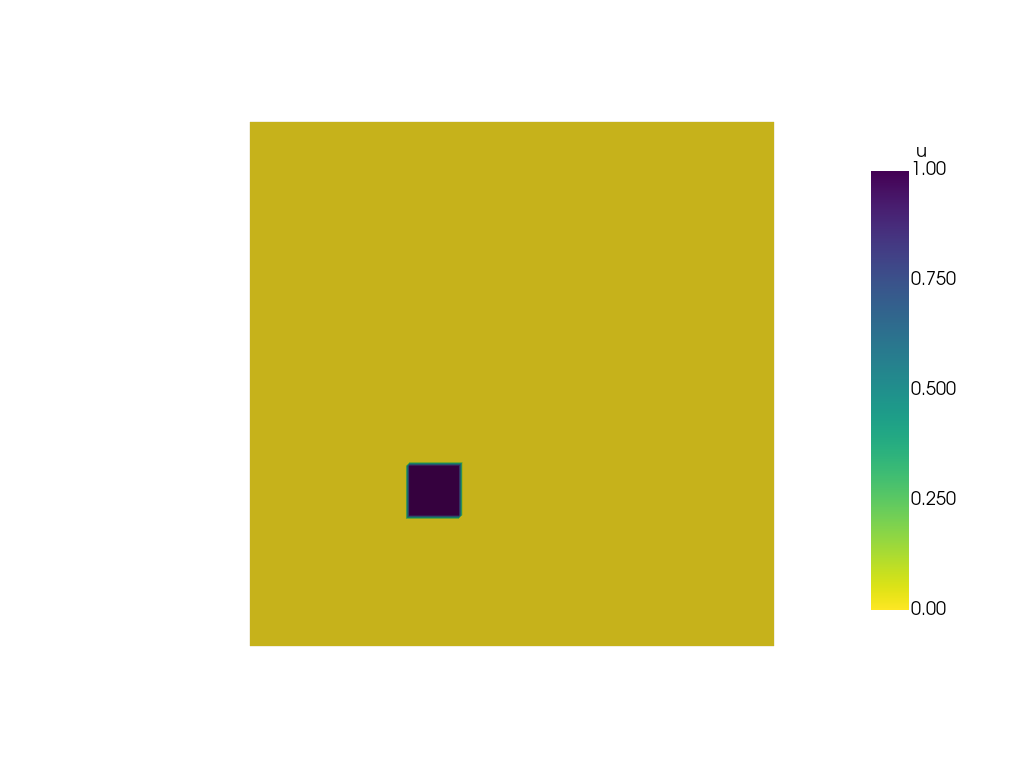}
        \vspace{-3em}
        \label{true_single}
        \caption{True source}
    \end{subfigure}
    \begin{subfigure}[b]{0.45\linewidth}        %% or \columnwidth
        \centering
        \includegraphics[width=\linewidth]{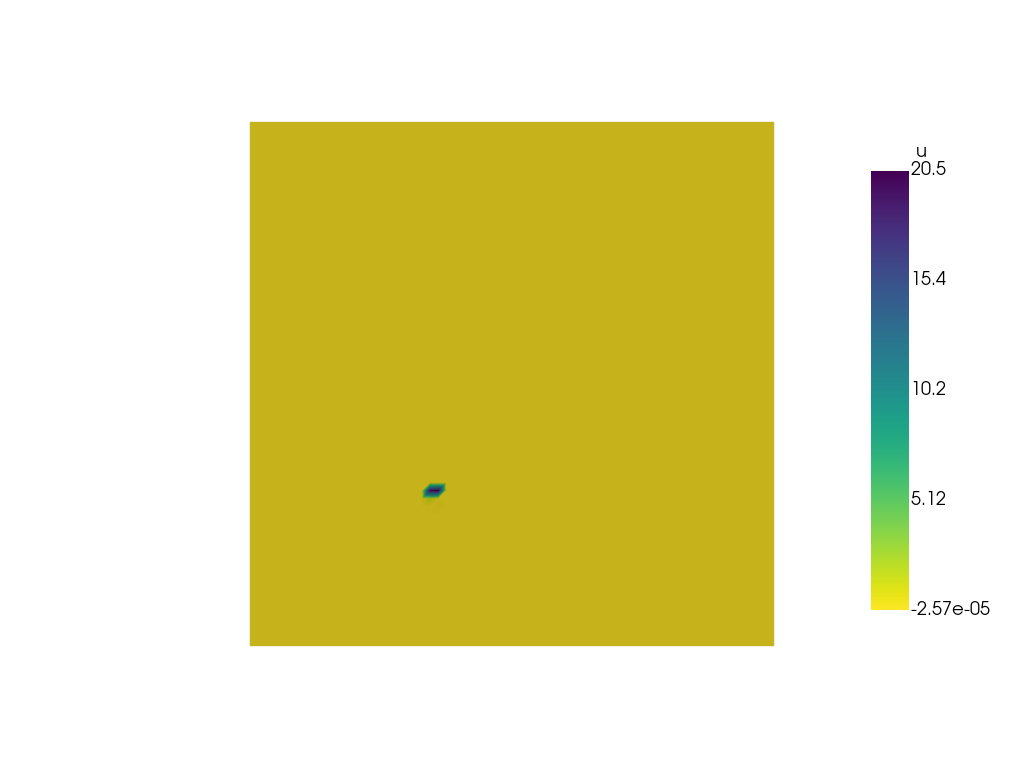}
        \vspace{-3em}
        \caption{$\BB = \II$}
    \end{subfigure}
    \begin{subfigure}[b]{0.45\linewidth}        %% or \columnwidth
        \centering
        \includegraphics[width=\linewidth]{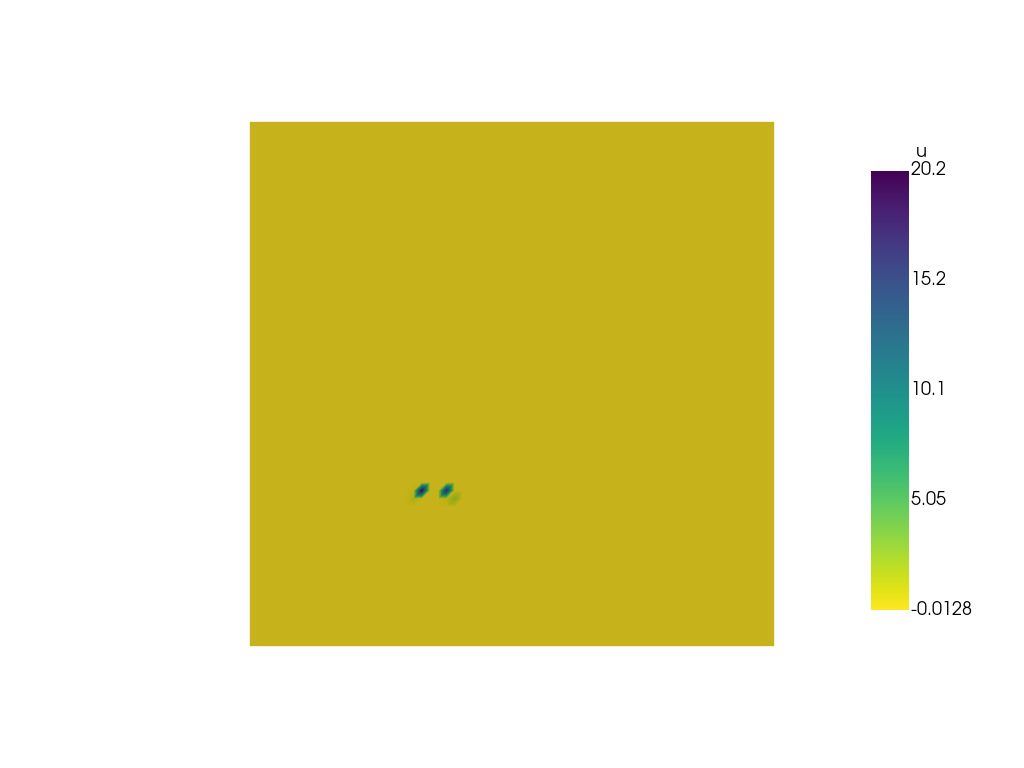}
        \vspace{-3em}
        \caption{$\BB = \AAK^\dagger$ }
    \end{subfigure}
    \begin{subfigure}[b]{0.45\linewidth}        %% or \columnwidth
        \centering
        \includegraphics[width=\linewidth]{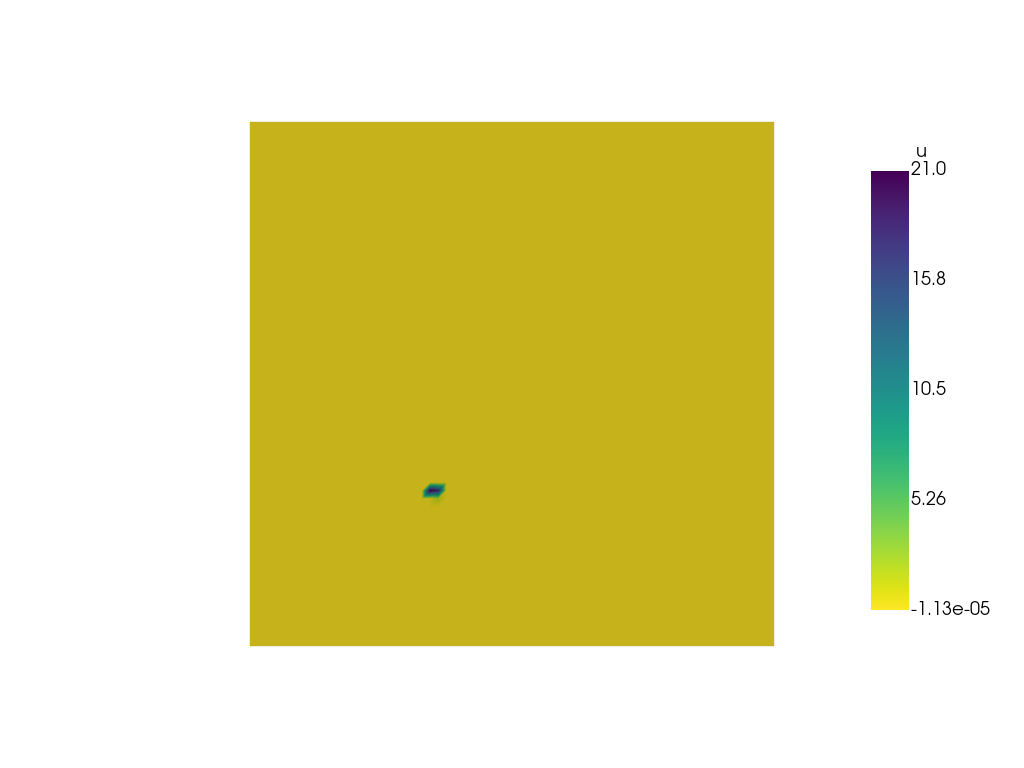}
        \vspace{-3em}
        \caption{$\BB  = \BB_{r} $  }
    \end{subfigure}
    
    \caption{Comparison of the true source and the inverse solutions computed in the case of a single source.}
    \label{one_source}
    
\end{figure}
 
We also considered a case where the true source is comprised of three spatially separated sources, for both the Helmholtz' equation ($\epsilon=-1$ in \eqref{eq: bdry_val_prob}) and the screened Poisson equation ($\epsilon=1$), as illustrated in figures \ref{helm} and \ref{scr_poi}, respectively. Figure \ref{helm} contains results computed with Gaussian noise added to the observation data, whereas Figure \ref{scr_poi} shows the outcome of noise-free simulations.

The noise vector $\boldsymbol{\eta}$ was generated from the normal distribution $\mathcal{N}(0,\sigma^2\mathsf{I})$ and then rescaled so that the ratio $\frac{\|\boldsymbol{\eta}\|_2}{\|\yy\|_2}$ between the noise-free data $\yy$ and the noise $\boldsymbol{\eta}$ was 0.02, that is, a noise level of 2\%.

Note that $\BB=\AAK^\dagger$ provides rather accurate recoveries, and that the true sources "collapse" to one source in the results generated with the other two choices of $\BB$ with noisy data and $\epsilon=-1$, see Figure \ref{helm}.

\begin{figure}[H]
    \centering
    \begin{subfigure}[b]{0.45\linewidth}        %% or \columnwidth
        \centering
        \includegraphics[width=\linewidth]{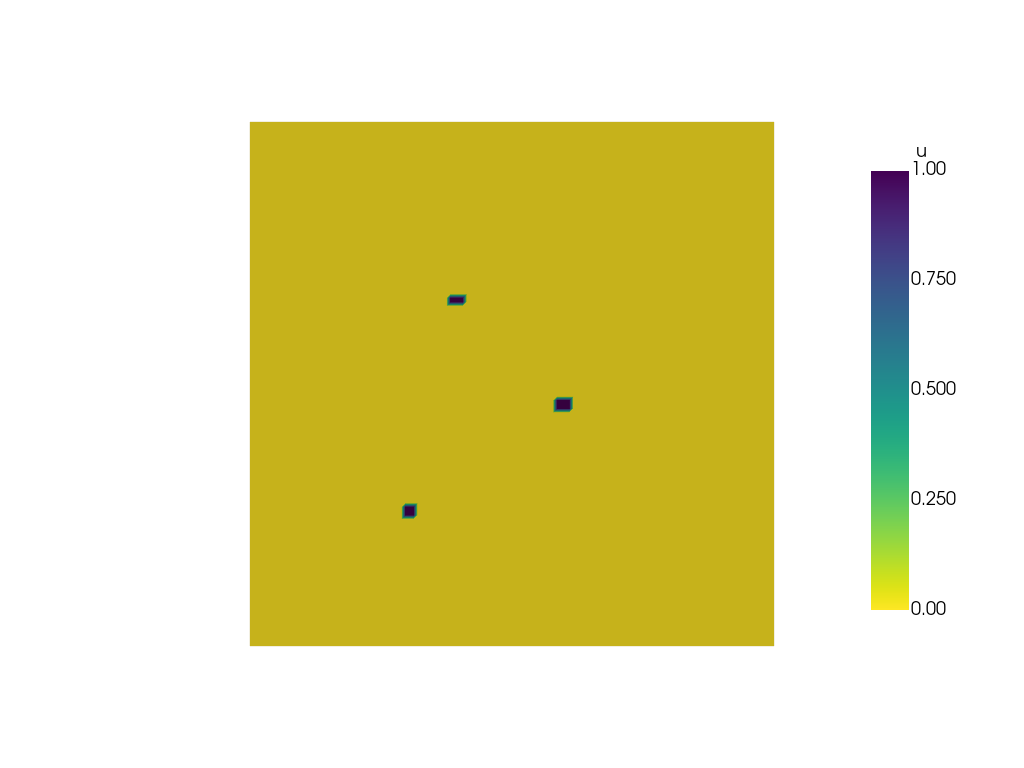}
        \vspace{-3em}
        \label{true_single}
        \caption{True sources}
    \end{subfigure}
    \begin{subfigure}[b]{0.45\linewidth}        %% or \columnwidth
        \centering
        \includegraphics[width=\linewidth]{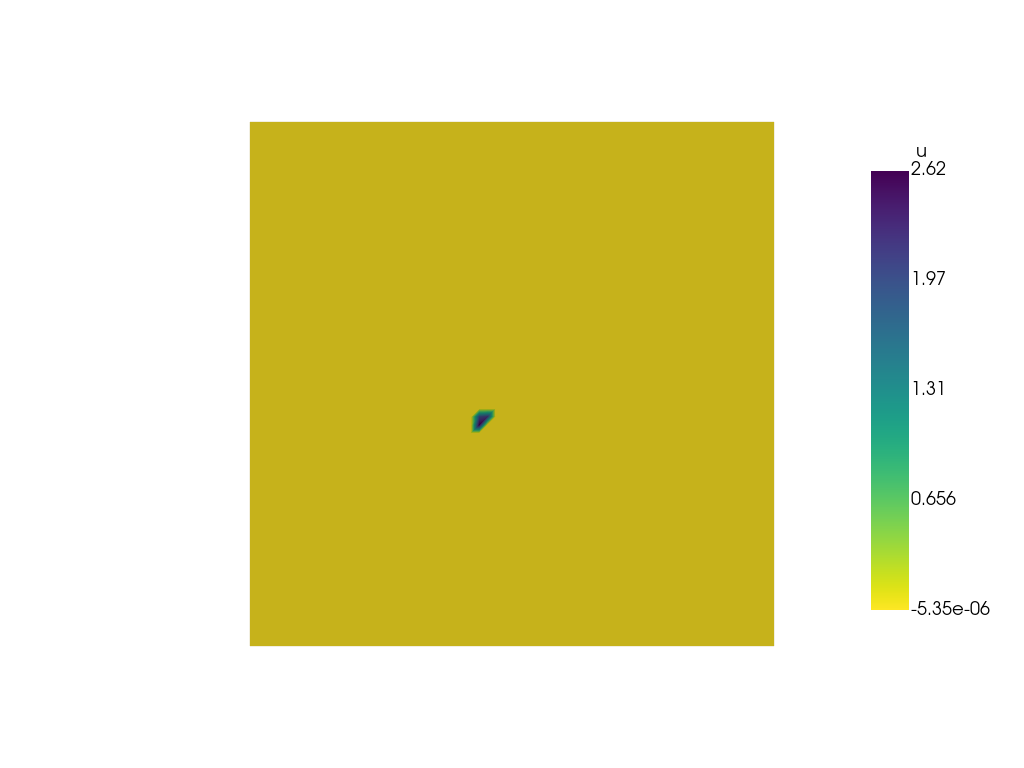}
        \vspace{-3em}
        \caption{$\BB = \II$}
    \end{subfigure}
    \begin{subfigure}[b]{0.45\linewidth}        %% or \columnwidth
        \centering
        \includegraphics[width=\linewidth]{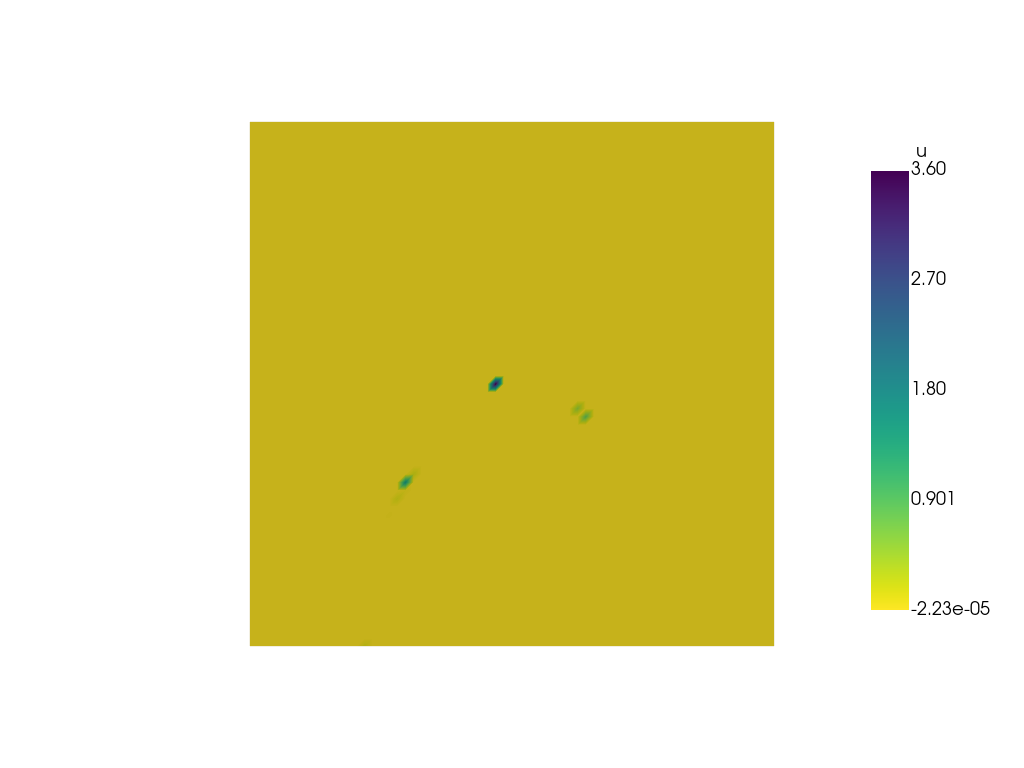}
        \vspace{-3em}
        \caption{$\BB = \AAK^\dagger$ }
    \end{subfigure}
    \begin{subfigure}[b]{0.45\linewidth}        %% or \columnwidth
        \centering
        \includegraphics[width=\linewidth]{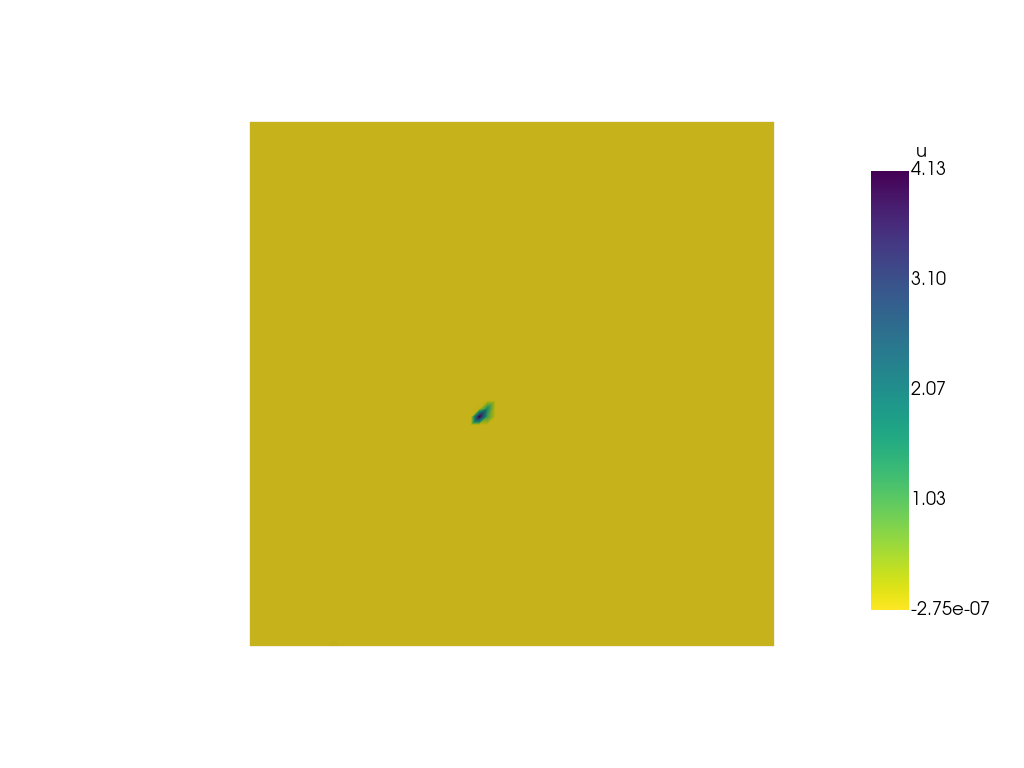}
        \vspace{-3em}
        \caption{$\BB  = \BB_{r} $  }
    \end{subfigure}
    
    \caption{Comparison of the true sources and the inverse solutions computed with $\epsilon = -1$ in \eqref{eq: bdry_val_prob}, i.e., with the Helmholtz' equation. Simulations with 2\% noise added to the (synthetic) observation data.}
    \label{helm}
    
\end{figure}

% Screened poison

\begin{figure}[H]
    \centering
    \begin{subfigure}[b]{0.45\linewidth}        %% or \columnwidth
        \centering
        \includegraphics[width=\linewidth]{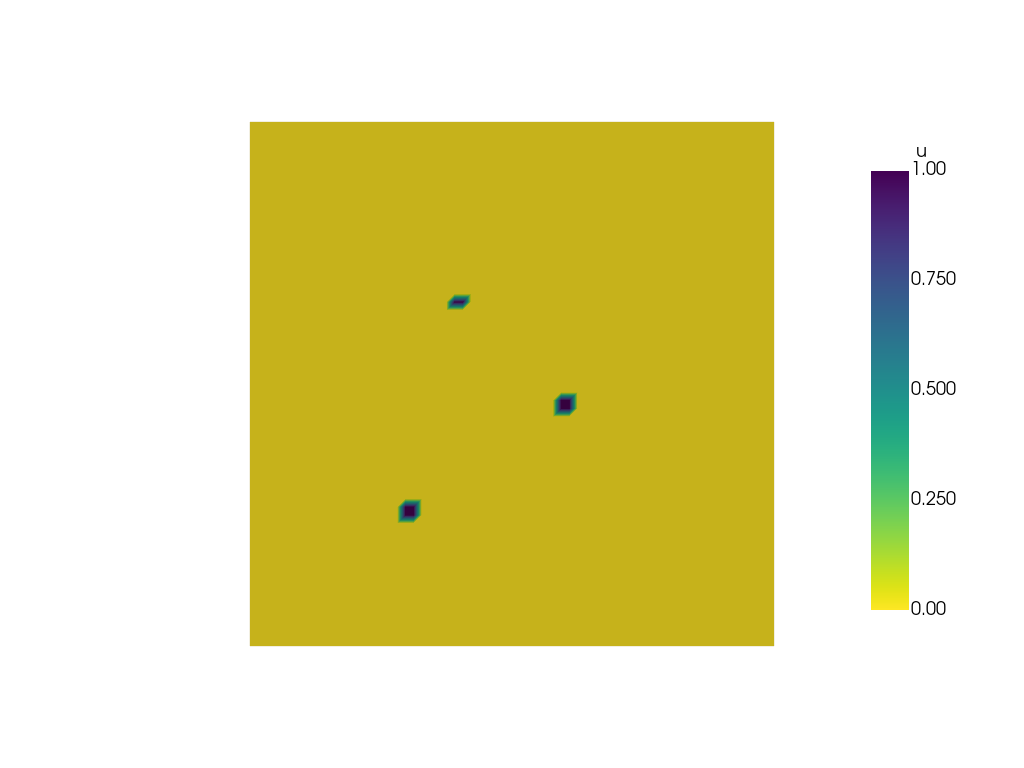}
        \vspace{-3em}
        \label{true_single}
        \caption{True sources}
    \end{subfigure}
    \begin{subfigure}[b]{0.45\linewidth}        %% or \columnwidth
        \centering
        \includegraphics[width=\linewidth]{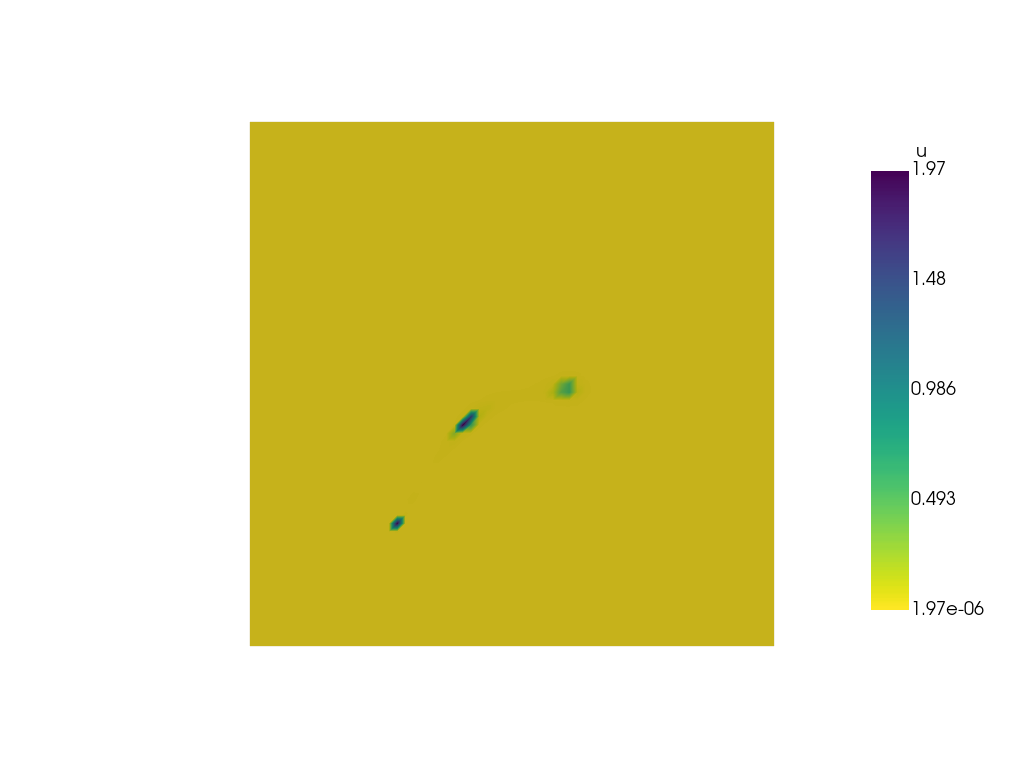}
        \vspace{-3em}
        \caption{$\BB = \II$}
    \end{subfigure}
    \begin{subfigure}[b]{0.45\linewidth}        %% or \columnwidth
        \centering
        \includegraphics[width=\linewidth]{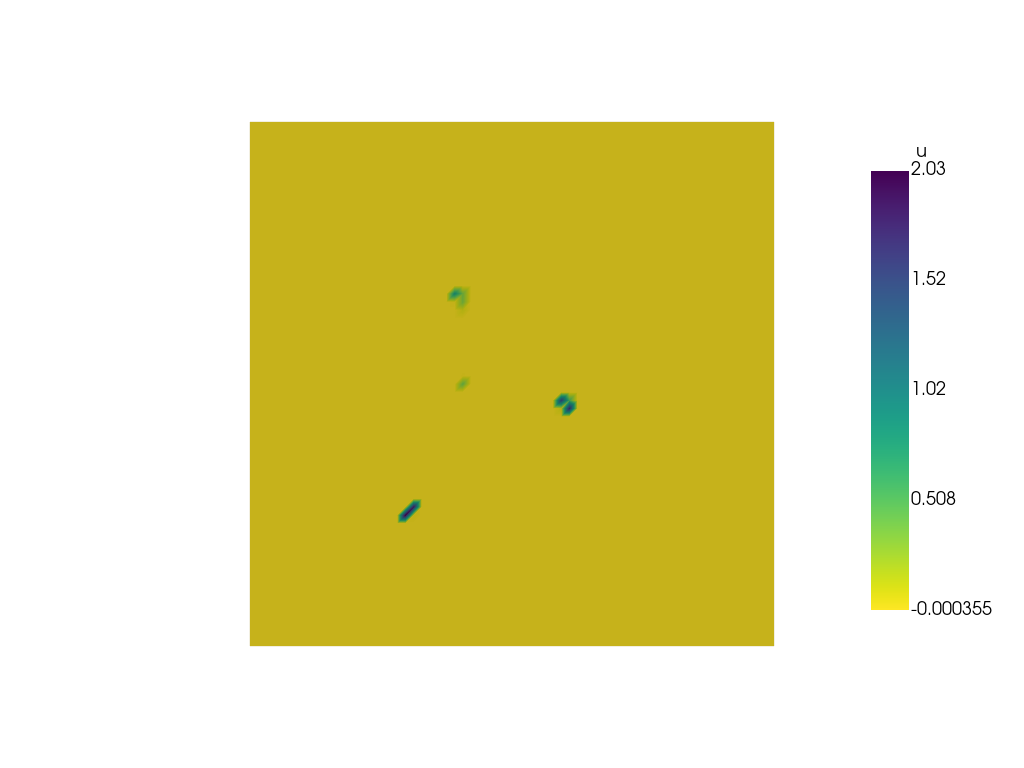}
        \vspace{-3em}
        \caption{$\BB = \AAK^\dagger$ }
    \end{subfigure}
    \begin{subfigure}[b]{0.45\linewidth}        %% or \columnwidth
        \centering
        \includegraphics[width=\linewidth]{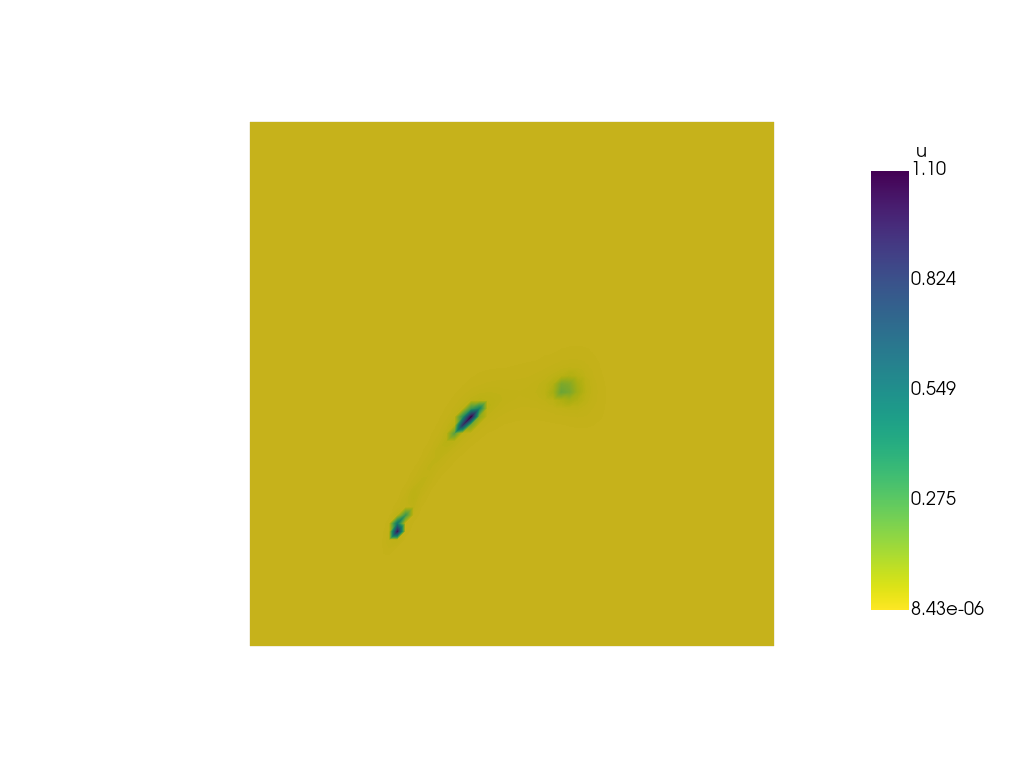}
        \vspace{-3em}
        \caption{$\BB  = \BB_{r} $  }
    \end{subfigure}
    
    \caption{Comparison of the true sources and the inverse solutions computed  with $\epsilon = 1$ in \eqref{eq: bdry_val_prob}, i.e., with the screened Poisson equation.} 
    \label{scr_poi}
    
\end{figure}

\rem{To connect the results of Figure \ref{scr_poi} with Theorem \ref{thm:disjoint_supports}, we quantify the violation of the disjointness condition \eqref{eq:disjoint_assumption} for each choice of B. The forward matrix, being a discretization of an elliptic PDE, produces vectors $\CC^T\CC\ee_j$ with global support. A meaningful measure of disjointness must therefore operate on a thresholded version of these vectors where small components are set to zero. Specifically, for each source $\ee_{j}$, associated with the three dots shown in Figure \ref{scr_poi}(a), we define the vector $\uu_j$ by
\begin{equation*}
    [\uu_j]_l =
    \begin{cases}
        |[\CC^T \CC \ee_j]_l|, & \text{if } |[\CC^T \CC \ee_j]_l| > \tau \|\CC^T \CC \ee_j\|_\infty, \\
        0, & \text{otherwise},
    \end{cases}
\end{equation*}
where $\tau \in [0,1]$ sets the threshold for nullifying components relative to $\|\CC^T \CC \ee_j\|_\infty$, and $[\uu_j]_l$ denotes the $l$'th component of $\uu_j$.
We can then compute the \textit{amount of (weak) disjointness} between $\CC^T \CC \ee_j$ and $\CC^T \CC \ee_k$ as the ratio $\tfrac{\nu}{n}$, where $\nu$ is the number of nonzero components that overlap between $\uu_j$ and $\uu_k$, and $n$ is the length of these vectors. In Figure \ref{support_plot}, $\uu_1$, $\uu_2$ and $\uu_3$ are associated with the upper most dot, the dot in the "center" and the lower most dot in Figure \ref{scr_poi}(a), respectively.}

 \rem{Figure \ref{support_plot} shows that when $\tau$ increases, the amount of overlap decays faster for $\BB = \AAK^\dagger$ than for $\BB = \II$ and $\BB = \BB_{r}$. This might, in light of Theorem \ref{thm:disjoint_supports}, explain why we obtained best results with $\BB =\AAK^\dagger$; see Figure \ref {scr_poi}.}
 
%\rem{Furthermore, as shown in Table \ref{tab:orthogonality}, the pairwise angles between the sources $\CC\ee_i$ and $\CC\ee_j$ are larger, i.e., closer to being orthogonal, for $B = \AAK^\dagger$ than for the other choices of $\BB$.}
 
\begin{figure}[H]
    \centering
    \begin{subfigure}[b]{0.6\linewidth}        
        \centering
    \includegraphics[width=\linewidth]{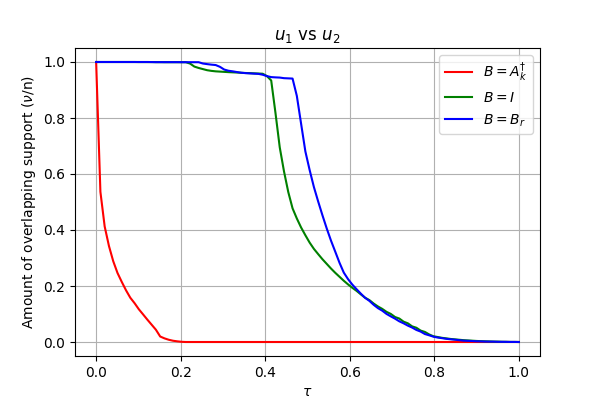}
        \vspace{-1em}
        \caption{}
        %\label{}
    \end{subfigure}\par
    \begin{subfigure}[b]{0.6\linewidth}        
        \centering
    \includegraphics[width=\linewidth]{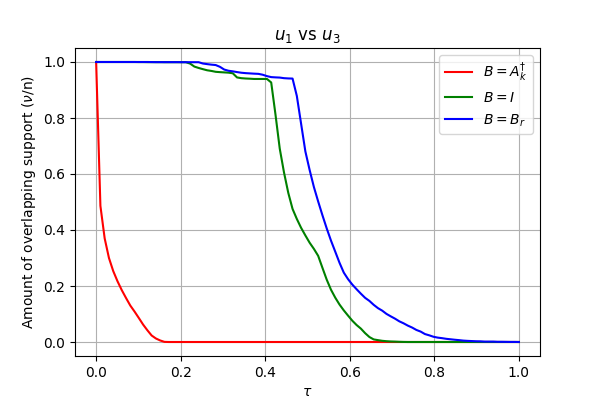}
        \vspace{-1em}
        \caption{}
        %\label{}
    \end{subfigure}\par
    \begin{subfigure}[b]{0.6\linewidth}       
        \centering
    \includegraphics[width=\linewidth]{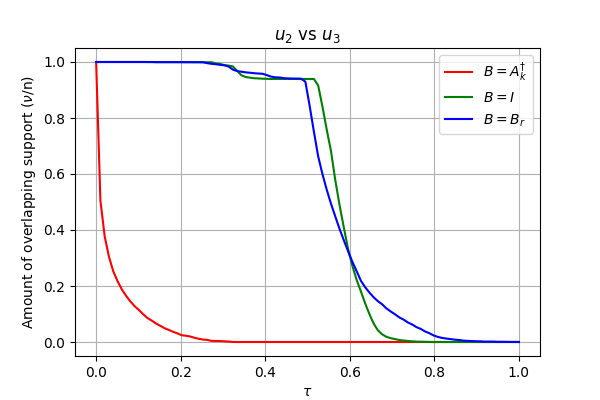}
        \vspace{-1em}
        \caption{}
        %\label{}
    \end{subfigure}
    \caption{\rem{The overlap ratio $\tfrac{\nu}{n}$ between the vectors $\uu_{j}$ and $\uu_{k}$ for different choices of $\BB$. %The indices $j_1$, $j_2$ and $j_3$ refers to the positions for each of the three different sources, cf. Figure \ref{scr_poi}a.
    }
    }
   \label{support_plot}
\end{figure}

\subsection{Almost parallel images of sources}
In this second example, the true source $\xx^*$ is defined in terms of three basis vectors associated with three adjacent grid cells. It is, without doing a deep analysis, reasonable to suspect that these three adjacent sources produce almost parallel images under the matrix $\CC$, making Theorem \ref{thm:almost_same_images} relevant for this case. 
 
To be in alignment with the theory, and as noted in the introduction to this section, we deliberately committed the inverse crime in the first part of this experiment by using the same grid for both the forward and inverse simulations.
In this idealized setting, we were able to almost perfectly recover the composite source consisting of these three adjacent basis vectors, see Figure \ref{coarser_multi}. This is in agreement with Theorem \ref{thm:almost_same_images}. 
 
\begin{figure}[H]
    \centering
    \begin{subfigure}[b]{0.45\linewidth}        %% or \columnwidth
        \centering
        \includegraphics[width=\linewidth]{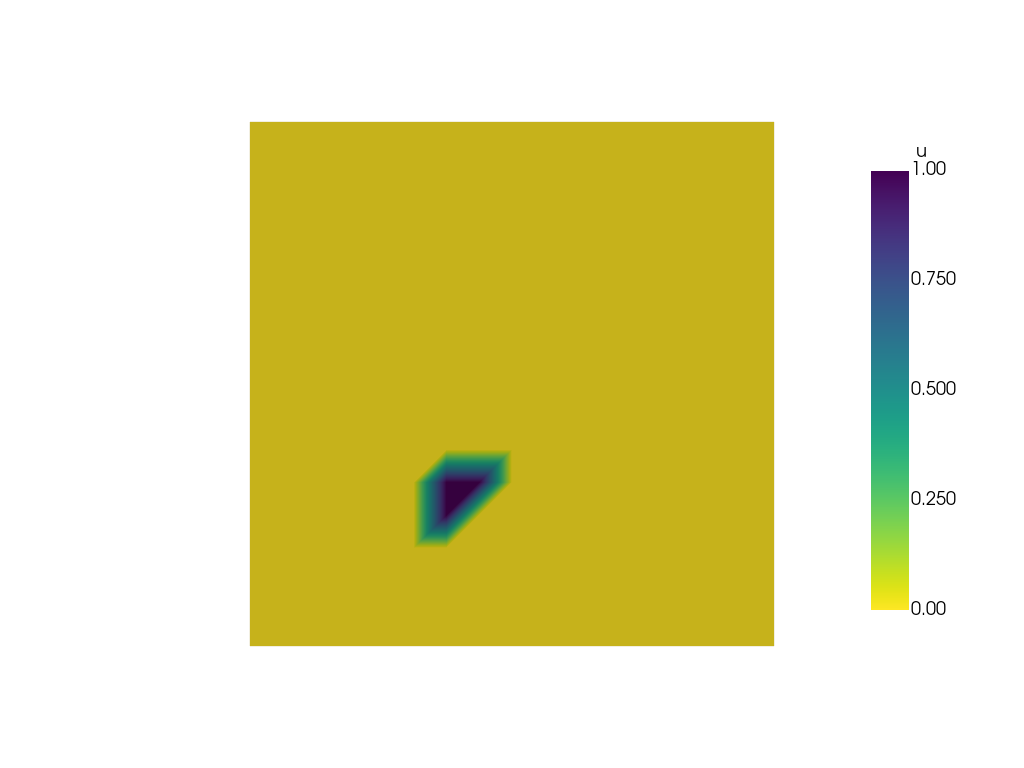}
        \vspace{-3em}
        \caption{True source}
    \end{subfigure}
    \begin{subfigure}[b]{0.45\linewidth}        %% or \columnwidth
        \centering
        \includegraphics[width=\linewidth]{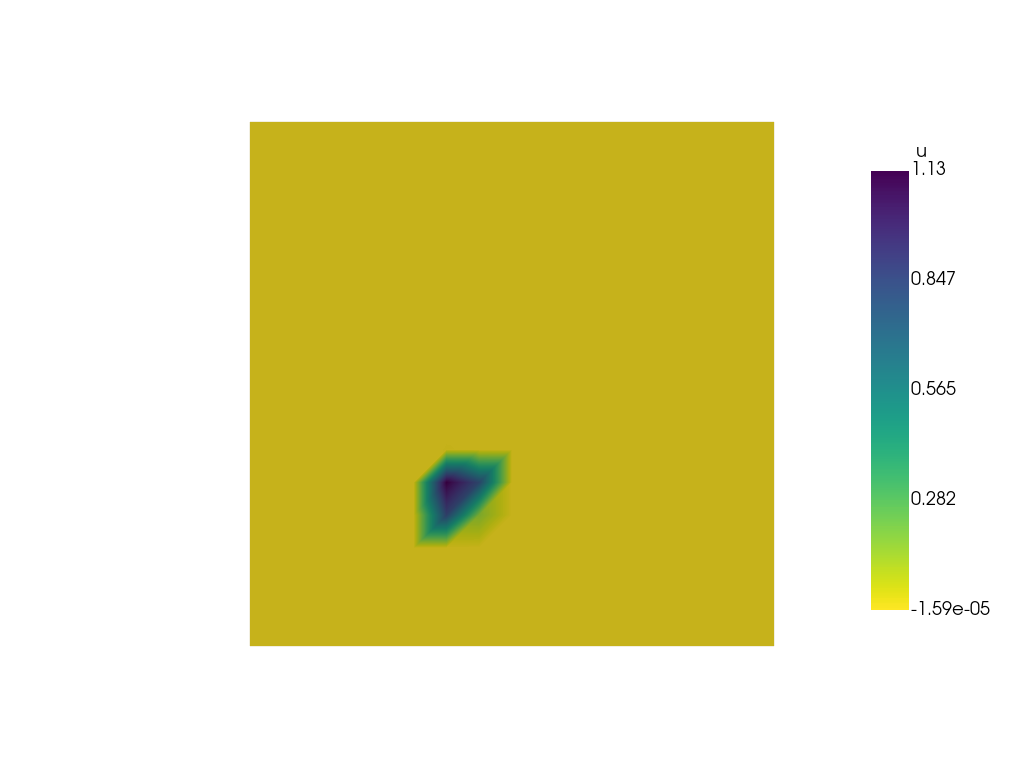}
        \vspace{-3em}
        \caption{$\BB = \II$}
    \end{subfigure}
    \begin{subfigure}[b]{0.45\linewidth}        %% or \columnwidth
        \centering
        \includegraphics[width=\linewidth]{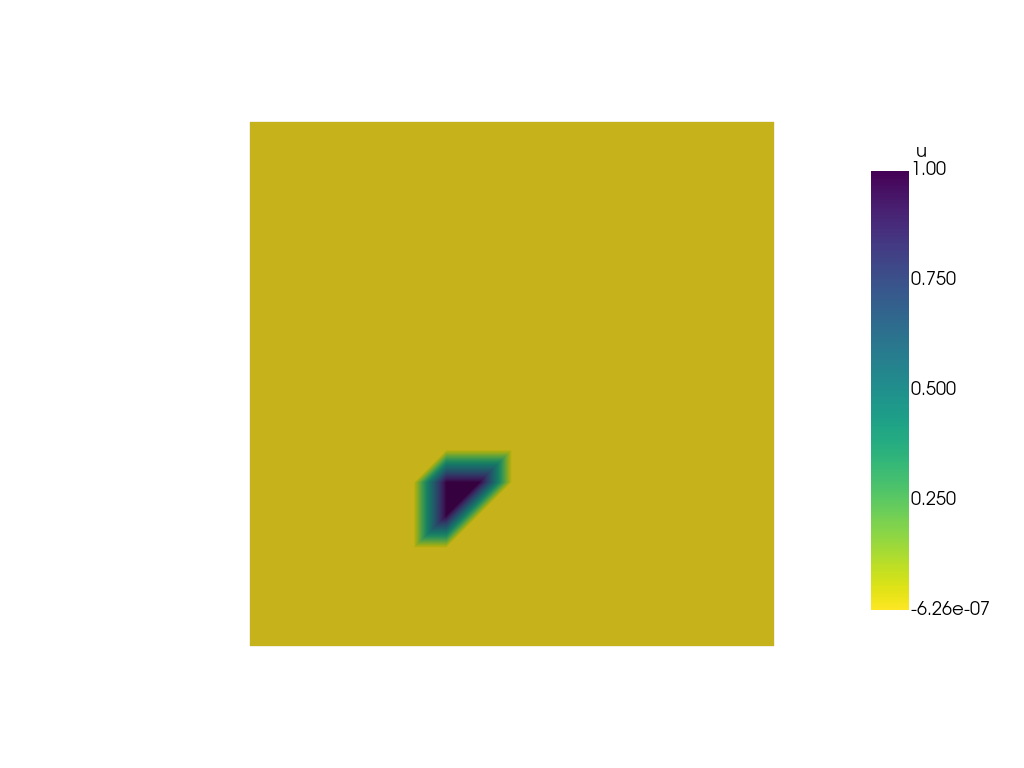}
        \vspace{-3em}
        \caption{$\BB = \AAK^\dagger$ }
    \end{subfigure}
    \begin{subfigure}[b]{0.45\linewidth}        %% or \columnwidth
        \centering
        \includegraphics[width=\linewidth]{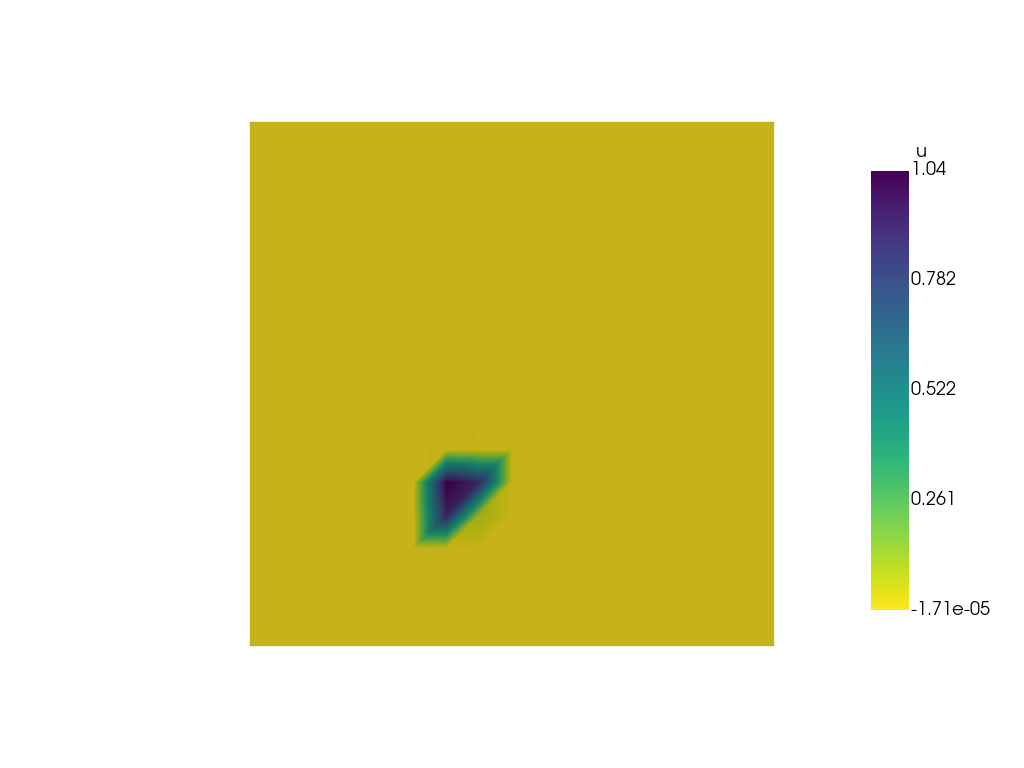}
        \vspace{-3em}
        \caption{$\BB  = \BB_{r} $  }
    \end{subfigure}
    
    \caption{Almost parallel images of the "sub-sources" constituting a composite source. The forward and inverse computations used the same $16 \times 16$ grid.}
    \label{coarser_multi}
    
\end{figure}

For the simulations shown in Figure \ref{finer_multi}, we avoided inverse crimes by using a finer mesh for the data generation than in the reconstruction process. In addition, as can be observed in panel (a) of Figure \ref{finer_multi}, we constructed a true solution with a small gap between its three "sub-sources", which will most likely lead to a violation of condition \eqref{eq:most_parallel}. As a result, the reconstruction schemes only produced single sources, but nevertheless the localization is correct. 
\begin{figure}[H]
    \centering
    \begin{subfigure}[b]{0.45\linewidth}        %% or \columnwidth
        \centering
        \hspace*{0.72cm}
        \includegraphics[width=0.68\linewidth]{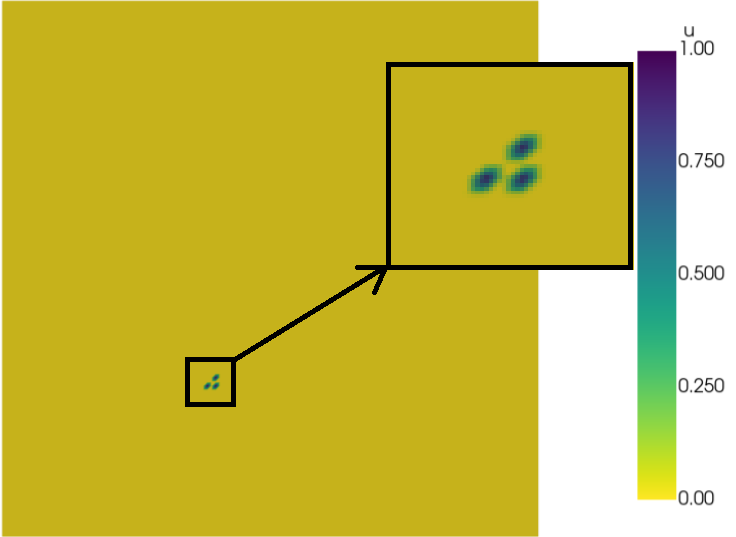}
        %\vspace{-3em}
        \caption{True source}
    \end{subfigure}
    \begin{subfigure}[b]{0.45\linewidth}        %% or \columnwidth
        \centering
        \includegraphics[width=\linewidth]{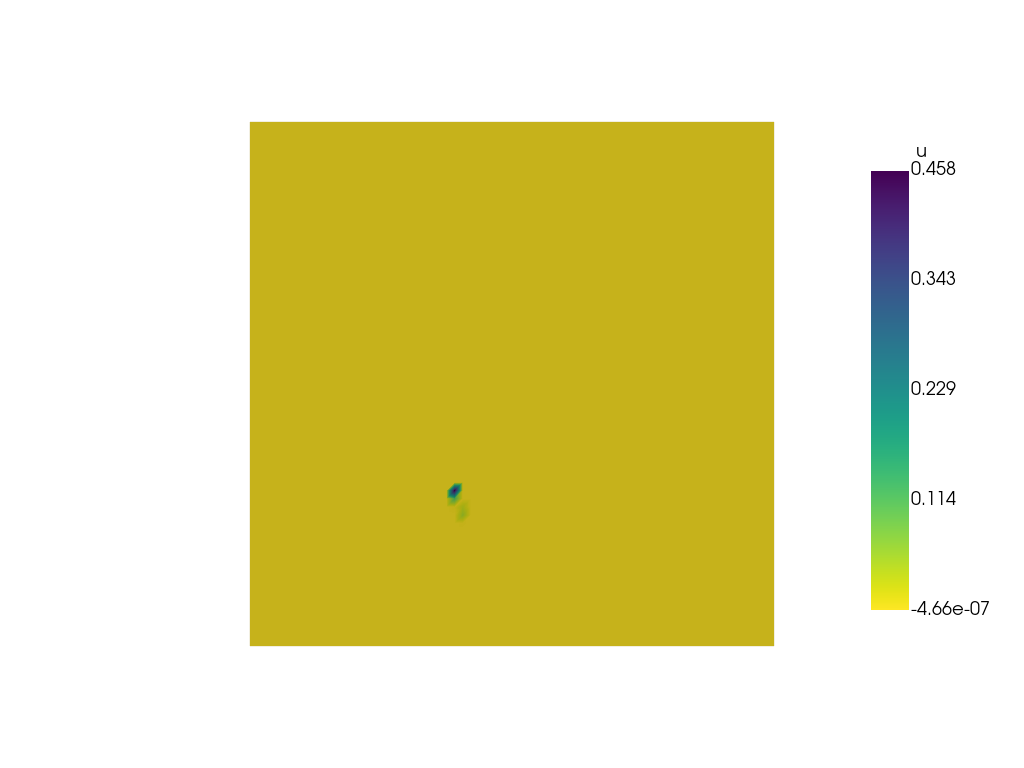}
        \vspace{-3em}
        \caption{$\BB = \II$}
    \end{subfigure}
    \begin{subfigure}[b]{0.45\linewidth}        %% or \columnwidth
        \centering
        \includegraphics[width=\linewidth]{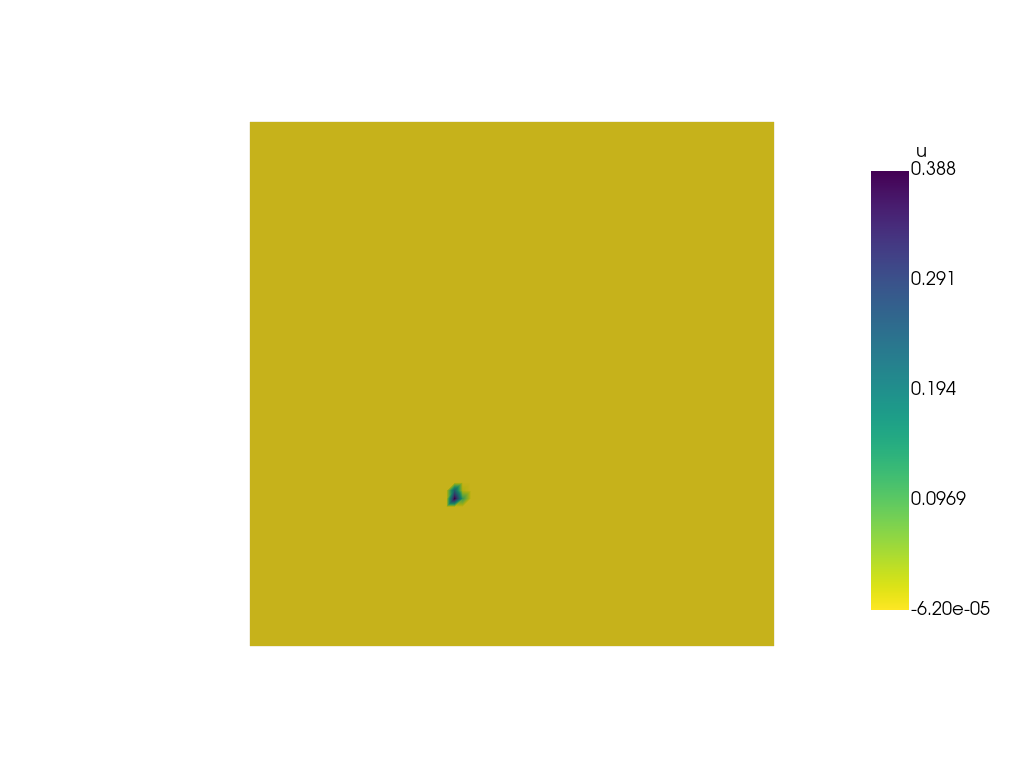}
        \vspace{-3em}
        \caption{$\BB = \AAK^\dagger$ }
    \end{subfigure}
    \begin{subfigure}[b]{0.45\linewidth}        %% or \columnwidth
        \centering
        \includegraphics[width=\linewidth]{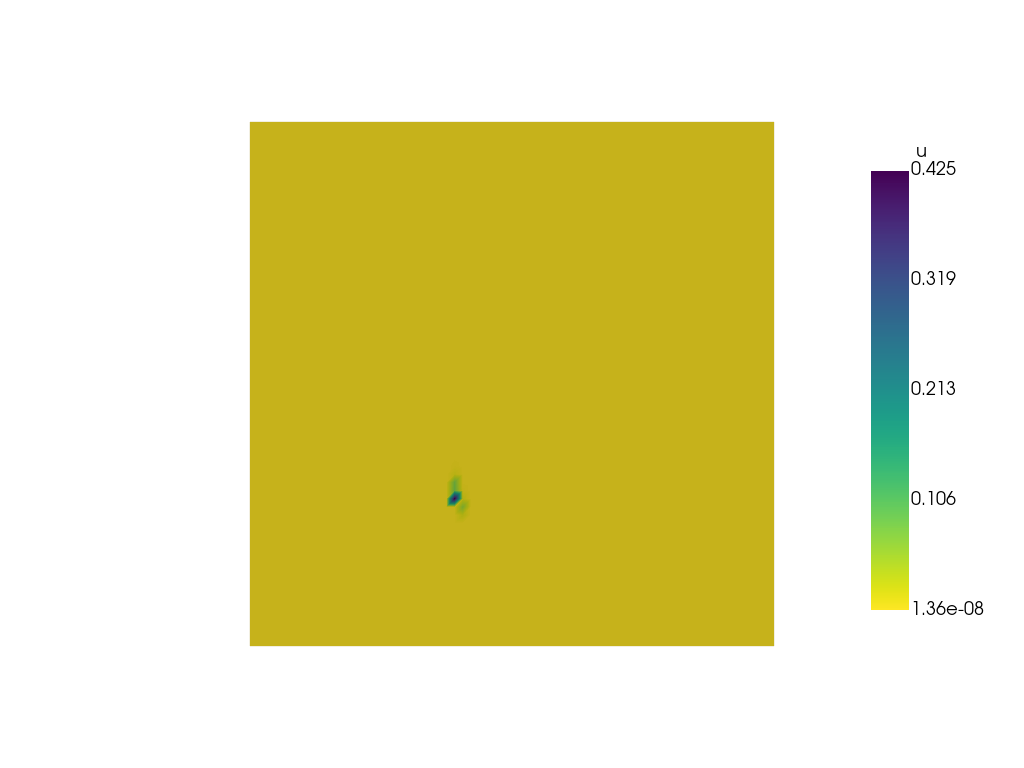}
        \vspace{-3em}
        \caption{$\BB  = \BB_{r} $  }
    \end{subfigure}
    
    \caption{Almost parallel images of the "sub-sources" constituting a composite source. Forward grid of size $128 \times 128$ and inverse grid of size $64 \times 64$.}
    \label{finer_multi}
    
\end{figure}

\subsection{Pre-orthogonalizer}
Proposition \ref{prop:new_B} guarantees that \rem{the orthogonality property \eqref{eq:orthogonality_assumption}, needed in the proof of Theorem \ref{thm:disjoint_supports},} can be satisfied by choosing $\BB = \YY^\dagger$, where $\YY$ is defined in \eqref{def_Y} and is constructed by selecting certain columns of $\AAA$. %In this example, we can see the choice of columns in Figure \ref{fig:YY}.
Figure \ref{fig:YY} illustrates this selection process for the present example: Each column of $\YY$ contains the forward image under $\AAA$ of a basis vector associated with one of the "dots"/grid-cells in this figure. (We refer to $\YY$ as a pre-orthogonalizer because its use ensures that the orthogonality \rem{\eqref{eq:orthogonality_assumption}} holds). 

It is not likely that condition \rem{\eqref{eq:disjoint_assumption}} in Theorem \ref{thm:disjoint_supports} \rem{also is} satisfied, but in this example we will nevertheless illustrate numerically how the choice $\BB = \YY^{\dagger}$ can potentially improve recoveries. 

As we can see in Figure \ref{fig:ex1_0_01_noise}, we consider a true configuration consisting of two sources and two sinks, which are well separated. We observe that the use of $ \YY^\dagger$ provides a more accurate recovery of the true sources and sinks compared with $\AAK^\dagger$.

\begin{figure}[H]
    \centering
    \includegraphics[width=0.7\linewidth]{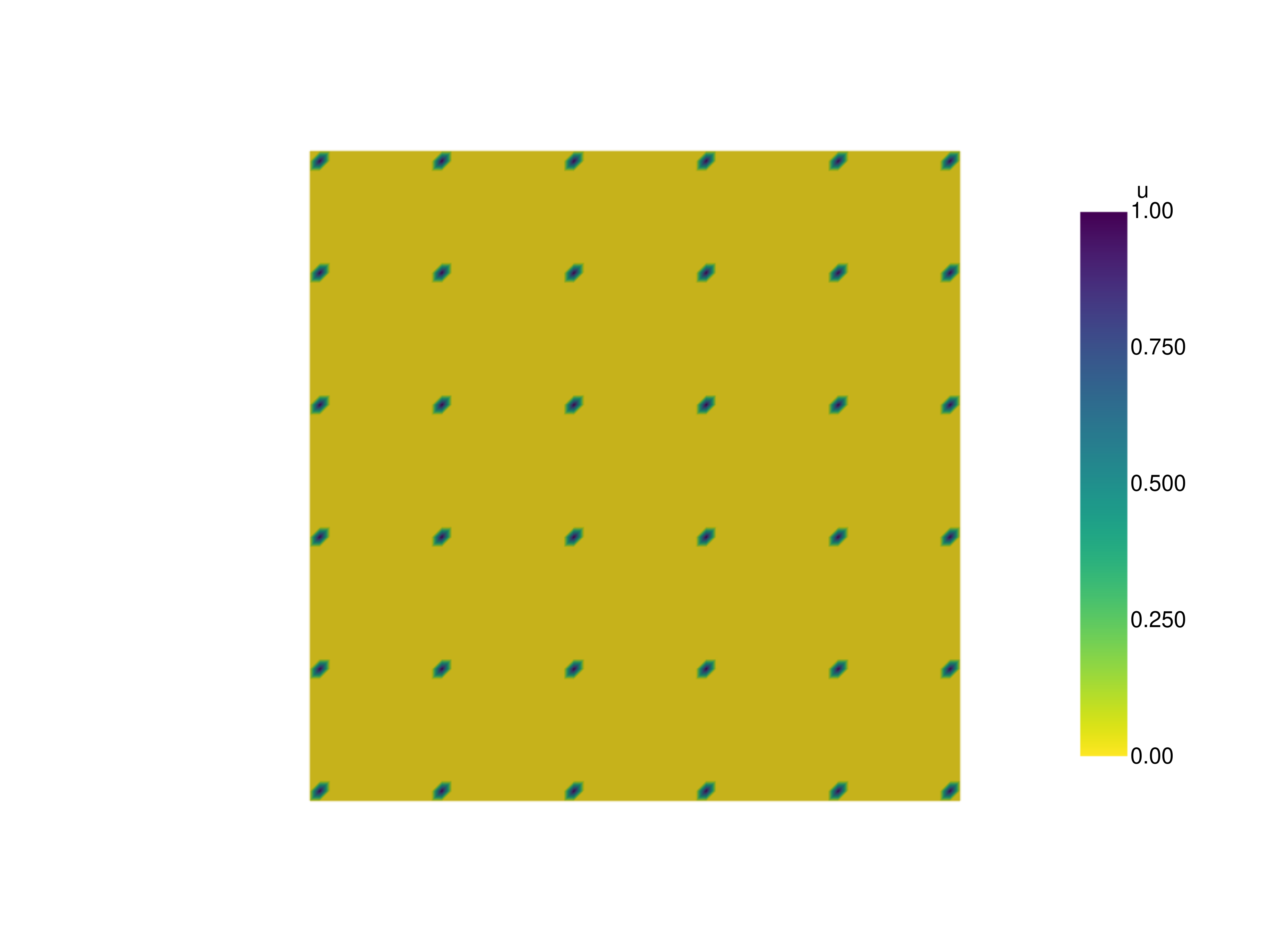}
    \vspace{-3em}
    \caption{Illustration of which subcolumns of $\AAA$ that is used to define $\YY$.}\label{fig:YY}
\end{figure}

\begin{figure}[H]
    \centering
    \begin{subfigure}[b]{0.6\linewidth}        %% or \columnwidth
        \centering
        \includegraphics[width=\linewidth]{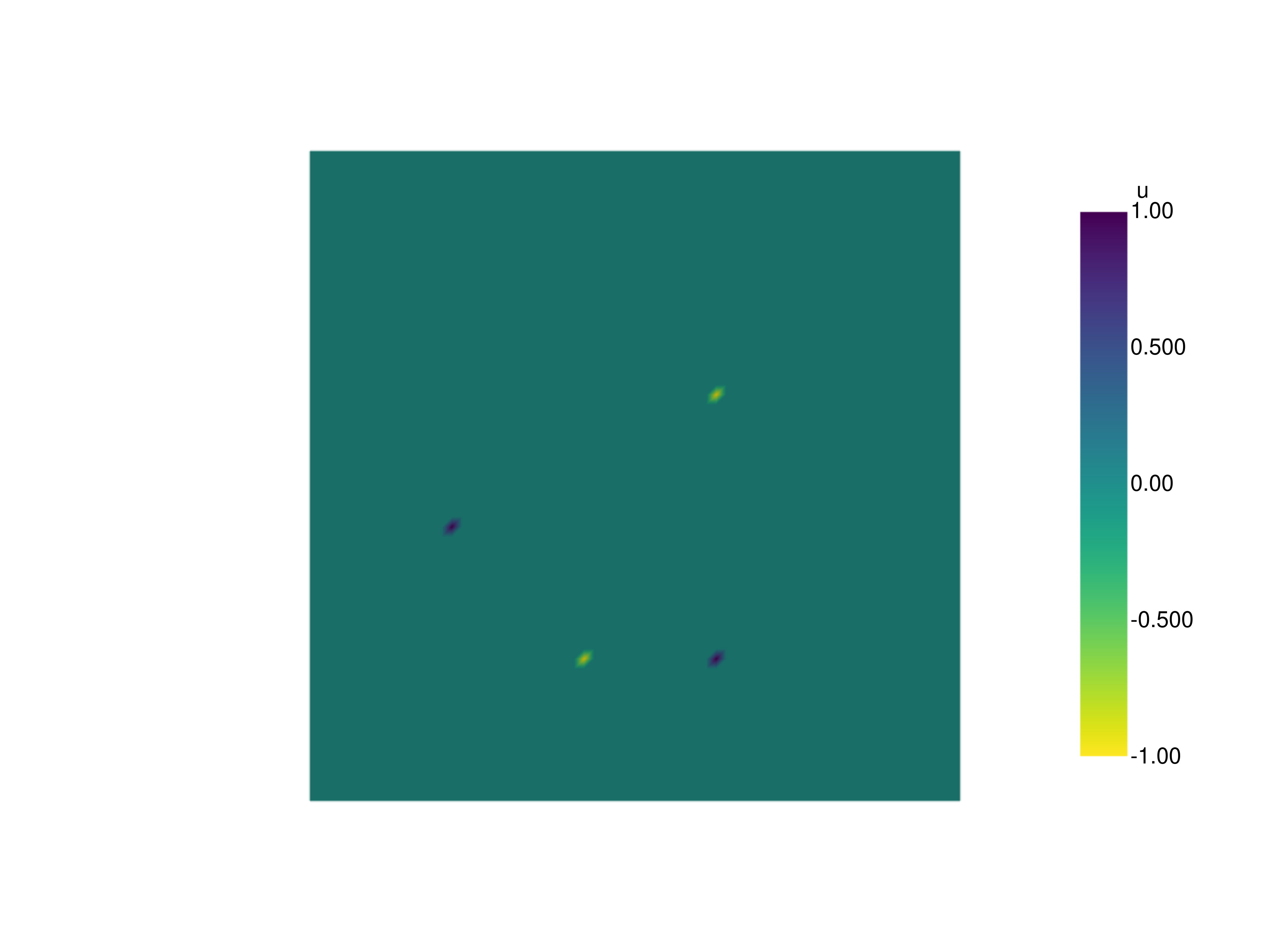}
        \vspace{-3em}
        \caption{True sources and sinks.}
    \end{subfigure}\par
    \begin{subfigure}[b]{0.6\linewidth}        %% or \columnwidth
        \centering
        \includegraphics[width=\linewidth]{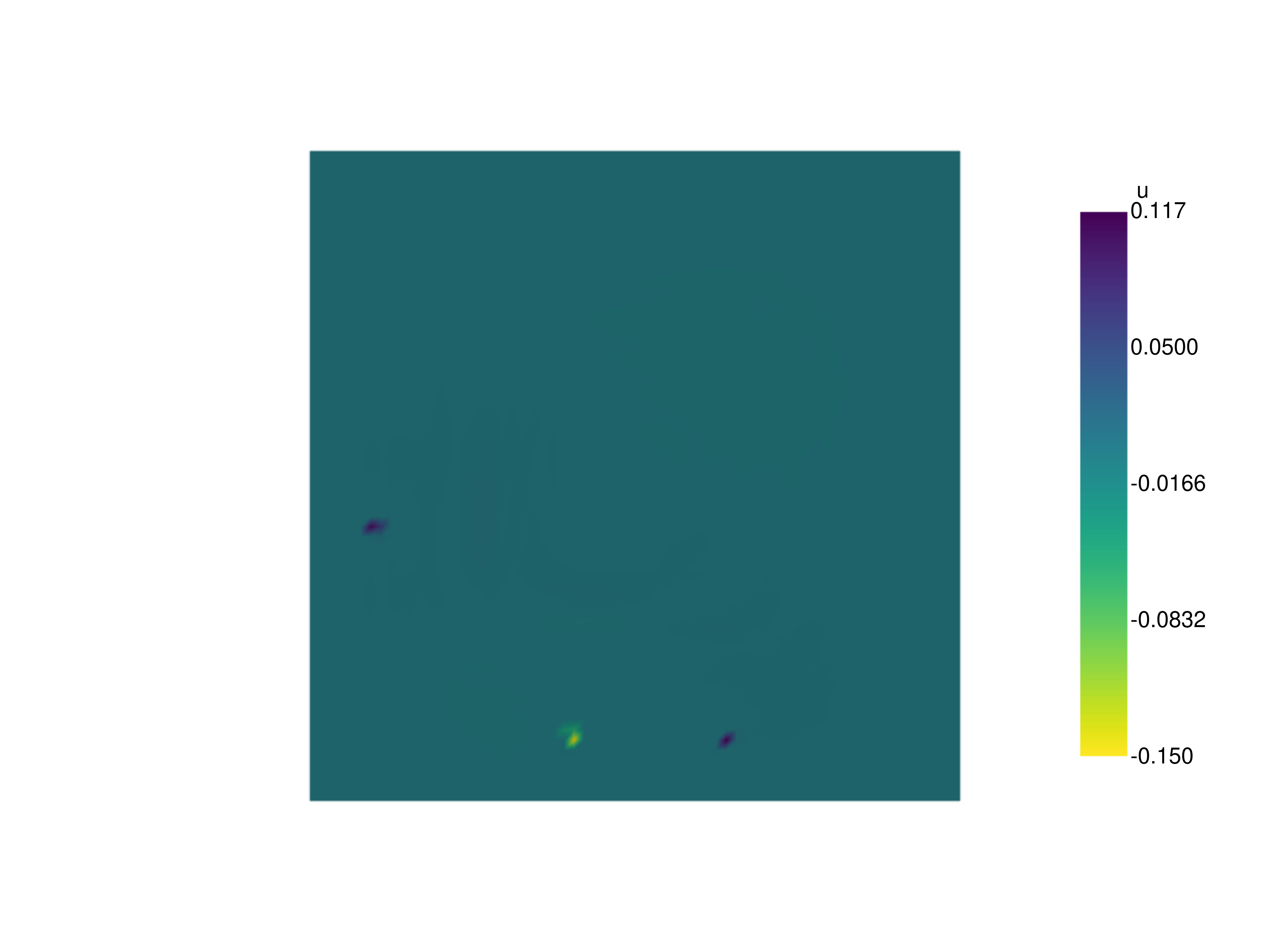}
        \vspace{-3em}
        \caption{$\BB = \AAK^\dagger$.}
    \end{subfigure}\par
    \begin{subfigure}[b]{0.6\linewidth}        %% or \columnwidth
        \centering
        \includegraphics[width=\linewidth]{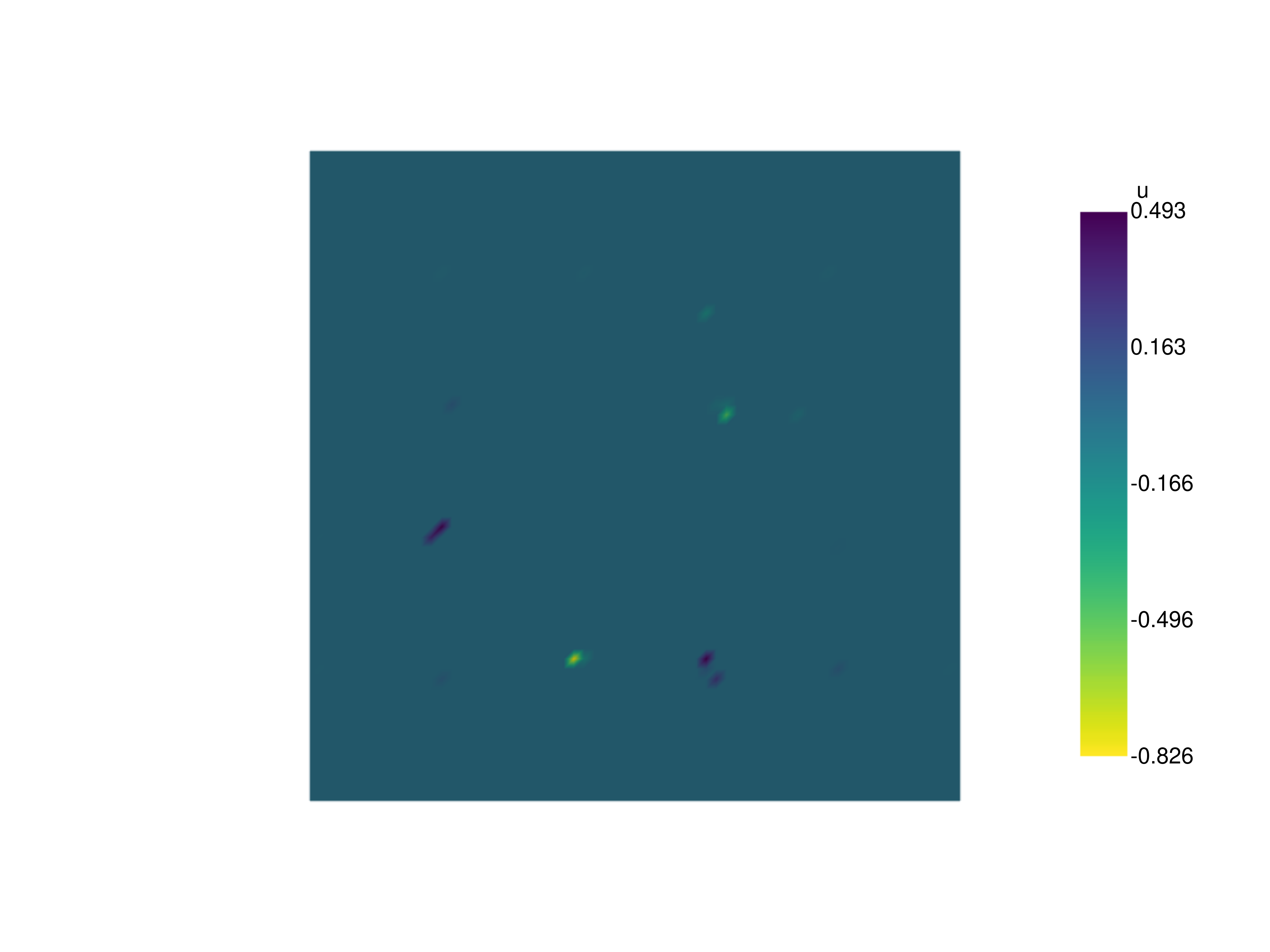}
        \vspace{-3em}
        \caption{$\BB = \YY^\dagger$.}
    \end{subfigure}
    \caption{True setup and inverse solutions computed with $\BB=\AAK^\dagger$ and $\BB=\YY^\dagger$.}
    \label{fig:ex1_0_01_noise}
\end{figure}
\section{Summary}

We have examined the effects of different auxiliary operators $\BB$, employed in the construction of the weight matrix $\WW$, for solving inverse source problems. Our results pave the way for the recovery of sources when the weights are constructed in terms of the mapping $\CC = \BB\AAA$, where $\AAA$ is the involved forward matrix, which typically has a large null space. Ranging from single to well-separated multiple sources, our findings provide strong evidence for the potential successful recovery of sources under certain assumptions.

The numerical experiments presented in this paper align with our theoretical results. In the case where the true source was constructed using a single basis vector, the inverse solution procedures worked very well. Moreover,  in some specific complex scenarios, such as the one involving three adjacent sources with almost parallel images, we could also ensure perfect recovery under specific assumptions, which we exemplified numerically. 

Across most of our experiments, choosing $\BB = \AAK^\dagger$, a truncated version of the pseudo inverse of $\AAA$, produced better results compared to the other choices of $\BB$ that we tested. However, in the case of multiple sources and sinks, we observed numerically that selecting certain columns of $\AAA$, leading to the operator $\BB=\YY^\dagger$, led to improved recovery compared with employing $\BB=\AAK^\dagger$. 

An interesting open question that arises from our research is how to select the weighting operator $\BB$ based on a specific problem setup, i.e., how to use certain properties of the forward matrix $\AAA$ to design $\BB$. Although the theory presented in this work ensures the recovery of sources under certain conditions, developing the theory further based on specific characteristics of $\AAA$ and $\BB$ may be of interest for future work.

\appendix 
\section{\large Proof of Theorem \ref{thm:varThm}} \label{proof_thm:varThm}
\begin{proof}
Using the notation $$h(\zz)=\| \zz \|_1,$$ the first order optimality condition for \eqref{eq:varform_simple} reads 
\begin{equation*}
    \OO \in \CC^T (\CC \xx-\CC \ee_j) + \alpha \WW \partial h(\WW \xx), 
\end{equation*}
where $\partial h$ denotes the subgradient of $h$. The involved cost-functional is convex, and this condition is thus both necessary and sufficient. Inserting $\xx = \gamma_\alpha \ee_j$ to the condition above, we obtain 
\begin{equation*}
    (1-\gamma_\alpha) \CC^T \CC \ee_j \in \alpha \WW \partial h(\gamma_\alpha \WW \ee_j),  
\end{equation*}
or, alternatively,
\begin{equation} \label{eq:fonc}
    \frac{(1-\gamma_\alpha)}{\alpha} \WW^{-1} \CC^T \CC \ee_j \in \partial h(\gamma_\alpha \WW \ee_j).   
\end{equation}
Since the entries of the diagonal matrix $\WW$ are strictly positive, it follows from standard computations that
\begin{equation*}
(\partial h(\gamma_\alpha \WW \ee_j),\ee_i) = \left\{ 
\begin{array}{cc}
   1,  & i=j, \\
    \left[ -1,1 \right],   & i \neq j, 
\end{array}
\right.
\end{equation*}
provided that $\gamma_\alpha > 0$. We thus may write \eqref{eq:fonc} in the following form 
\begin{equation}
\label{fonc_form}
    \frac{(1-\gamma_\alpha)}{\alpha} (\WW^{-1} \CC^T \CC \ee_j, \ee_i) \in 
    \left\{
    \begin{array}{cc}
   1,  & i=j, \\
    \left[ -1,1 \right],   & i \neq j. 
    \end{array}
    \right.
\end{equation}
Invoking \eqref{eq:maximum_expression} we therefore obtain the requirement  
\begin{equation} \label{eq:fonc_final}
    \frac{(1-\gamma_\alpha)}{\alpha} \|\CC\ee_j\| \left(\frac{\CC\ee_j}{\|\CC\ee_j\|}, \frac{\CC\ee_i}{\|\CC\ee_i\|}\right)  \in 
    \left\{
    \begin{array}{cc}
   1,  & i=j, \\
    \left[ -1,1 \right],   & i \neq j. 
    \end{array}
    \right.
\end{equation}
With the choice
\begin{equation*}
    \gamma_\alpha=1 - \frac{\alpha}{\wj},
\end{equation*}
it follows that \eqref{eq:fonc_final} holds for $i=j$, by recalling that $\wj=\|\CC\ee_j\|$. 
From Cauchy–Schwarz' inequality, we observe that \eqref{eq:fonc_final} also is satisfied when $i \neq j$, proving existence of the minimizer.
\vspace{5mm}

To show uniqueness, we first denote the cost-functional by $\mathfrak{J}$, i.e.,
\begin{equation*} 
  \mathfrak{J(\xx)} =  \frac{1}{2}\|\CC\xx - \CC\ee_j\|^2 + \alpha \|\WW\xx\|_1.
\end{equation*}
Let $\yy \in \mathbb{R}^n , \yy \neq \xx_{\alpha} $ be arbitrary. We will show that no such $\yy$ can be a minimizer, i.e., the minimizer is unique.
We split the analysis into two cases: \\ \\
\noindent
\textbf{Case 1: $\yy = c\xx_{\alpha}, c \neq 1$.} \\
By the convexity of the cost-functional in \eqref{eq:varform_simple} and the argument presented above, it follows that $\yy = c\xx_{\alpha}$ cannot be a minimizer unless $c = 1$.\\ \\ 
\noindent
\textbf{Case 2: $\yy \neq c\xx_{\alpha}$.} \\
In this case there must exist at least one component $y_k, k \neq j$, of $\yy$ such that  $y_k \neq 0 $. 
Consider
\begin{equation*}
  \mathfrak{J(\yy)} - \mathfrak{J(\xx_{\alpha})} = \frac{1}{2}\|\CC\yy - \CC\ee_j\|^2 - \frac{1}{2}\|\CC\xx_\alpha - \CC\ee_j\|^2 + \alpha\left(\|\WW\yy\|_1 - \|\WW\xx_\alpha\|_1\right).  
\end{equation*}
Also, by the definition of the subdifferential,
\begin{equation*}
    h(\WW\yy) - h(\WW\xx_\alpha) \geq \zz^T (\WW\yy - \WW\xx_\alpha)
\end{equation*}
for any $\zz \in \partial h(\WW\xx_\alpha)$. Consequently, we get
\begin{align}
     \mathfrak{J(\yy)} - \mathfrak{J(\xx_{\alpha})} &= \frac{1}{2}\|\CC\yy - \CC\ee_j\|^2 - \frac{1}{2}\|\CC\xx_\alpha - \CC\ee_j\|^2 \nonumber \\ & + \alpha\left( h(\WW\yy) - h(\WW\xx_\alpha) \right) \nonumber \\ &
     \geq \frac{1}{2}\|\CC\yy - \CC\ee_j\|^2 - \frac{1}{2}\|\CC\xx_\alpha - \CC\ee_j\|^2 \nonumber \\ & + \alpha\zz^T (\WW\yy - \WW\xx_\alpha)  \label{eq:costdiff}
\end{align}
Recall that $\xx_\alpha = \gamma_\alpha \ee_j$.
From Lemma \eqref{lemma} , we can write \eqref{fonc_form} as
\begin{align}
    \frac{1}{\alpha}\left(\WW^{-1}\CC^T \CC (\ee_j - \xx_{\alpha}), \ee_i\right) &\in 
    \label{z_value}
    \left\{
    \begin{array}{cc}
       1,  & i=j, \\
       \left( -1,1 \right),   & i \neq j, 
    \end{array}
    \right.\\
    &\subset 
    \left\{
       \begin{array}{cc}
       1,  & i=j, \\
       \left[ -1,1 \right],   & i \neq j. 
    \end{array}
    \right.\\
    & = \left( \partial h (\WW \xx_{\alpha}), \ee_i \right)
\end{align}
This implies that
\begin{equation}
\label{z_tilde_value}
     \frac{1}{\alpha} \WW^{-1}\CC^T \CC (\ee_j - \xx_{\alpha}) \in  \partial h (\WW \xx_{\alpha}) .
\end{equation}
However, choosing $\zz = \frac{1}{\alpha}  \WW^{-1}\CC^T \CC (\ee_j - \xx_{\alpha})$ does not immediately lead to a strict inequality in \eqref{eq:costdiff}. Consequently, we must find a better choice of $\zz$. Without loss of generality\footnote{If rather $[\WW\yy - \WW\xx_\alpha]_k < 0$ we could simply choose $\tilde{z}_k = -1$ and proceed in a similar fashion.}, we can assume that $ \left[ \WW\yy - \WW\xx_{\alpha}\right]_{k}>0 $ and choose $\tilde{\zz}= \left[\tilde{z}_1, \tilde{z}_2, \dots, \tilde{z}_n\right]^T $, where $\tilde{z}_i$ is defined as:
\begin{align*}
   \tilde{z}_i  &=
    \left\{
    \begin{array}{cc}
       1,  & i=k, \\
       \frac{1}{\alpha}\left(\WW^{-1}\CC^T \CC (\ee_j - \xx_{\alpha}), \ee_i\right),   & i \neq k, 
    \end{array}
    \right.\\ 
\end{align*}
Since the condition \eqref{z_tilde_value} holds, it follows that $\tilde{\zz} \in \partial h (\WW \xx_{\alpha}).$

From \eqref{z_value} we have
$[\frac{1}{\alpha}\WW^{-1}\CC^T \CC (\ee_j - \xx_{\alpha})]_{k} < 1 $ and therefore we get the strict inequality
  \begin{equation*}
      \tilde{\zz}^T (\WW\yy - \WW\xx_\alpha) > \frac{1}{\alpha}\WW^{-1}\CC^T \CC (\ee_j - \xx_{\alpha})^T (\WW\yy - \WW\xx_\alpha).
  \end{equation*}
Finally, combining this inequality with \eqref{eq:costdiff} we obtain
  \begin{align*}
     \mathfrak{J}(\yy) - \mathfrak{J}(\xx_{\alpha})  &
     \geq \frac{1}{2}\|\CC\yy - \CC\ee_j\|^2 - \frac{1}{2}\|\CC\xx_\alpha - \CC\ee_j\|^2 \\
     &\quad + \alpha\tilde{\zz}^T (\WW\yy - \WW \xx_\alpha) \\
     & > \frac{1}{2}\|\CC\yy - \CC\ee_j\|^2 - \frac{1}{2}\|\CC\xx_\alpha - \CC\ee_j\|^2 \\
     &\quad + \left(\WW^{-1}\CC^T \CC (\ee_j - \xx_{\alpha})\right)^T (\WW\yy - \WW\xx_\alpha)\\
    &= \frac{1}{2}\|\CC\yy - \CC\ee_j\|^2 - \frac{1}{2}\|\CC\xx_\alpha - \CC\ee_j\|^2 \\
    &\quad + \left(\CC^T \CC (\ee_j - \xx_{\alpha})\right)^T (\yy - \xx_\alpha)\\  
    & \geq 0,
\end{align*}
where the final inequality follows from the first-order optimality conditions of the convex functional $g(\xx) =\frac{1}{2}\|\CC\xx-\CC\ee_j\|^2$, i.e.,
\begin{eqnarray*}
    g(\yy) - g(\xx_\alpha) &\geq& \nabla g(\xx_\alpha)^T(\yy-\xx_\alpha) \\ &=& \left(\CC^T \CC (\xx_\alpha - \ee_j)\right)^T (\yy - \xx_\alpha).
\end{eqnarray*}
This shows that $\xx_{\alpha} $ is the unique minimizer of $\mathfrak{J(\xx)}$.
\end{proof}

\section{\large Proof of Theorem \ref{thm:disjoint_supports}} \label{proof_thm:disjoint_supports}
\begin{proof}
     Let 
\begin{equation*}
 \cc = \sum_{j \in \mathcal{J}} \text{sgn}(x_j^*) \frac{\CC \ee_j}{\|\CC \ee_j\|}.
 \label{eqn:c}
\end{equation*}
If we can show that \eqref{cond_1} and \eqref{cond_2} hold for this choice of $\cc$, the \rem{Theorem \ref{thm:disjoint_supports}} will follow immediately from Theorem \ref{thm:support}.

For $ i \in \mathcal{J} $, we have from the orthogonality \rem{\eqref{eq:orthogonality_assumption} of} $\{\CC\ee_j\}_{j\in \JJ}$ that
\begin{equation*}
  \frac{\CC \ee_i}{\|\CC \ee_i\|} \cdot \cc = \frac{\CC \ee_i}{\|\CC \ee_i\|} \cdot \frac{\CC \ee_i}{\|\CC \ee_i\|} \text{sgn}(x_i^*) = \text{sgn}(x_i^*),  
\end{equation*}
which shows that \eqref{cond_1} holds.

For $ i \in \mathcal{J}^c $, the support assumption \rem{\eqref{eq:disjoint_assumption}} implies that we have at most one $k \in \mathcal{J}$ such that $i \in supp(\CC^T\CC\ee_k)$. Consequently,
\begin{eqnarray}
  \frac{\CC \ee_i}{\|\CC \ee_i\|} \cdot \cc &=& \sum_{j \in \mathcal{J}} \text{sgn}(x_j^*) \frac{\CC \ee_i \cdot \CC \ee_j}{\|\CC \ee_i\| \|\CC \ee_j\|} \nonumber\\ 
 &=& \sum_{j \in \mathcal{J}} \text{sgn}(x_j^*) \frac{ \ee_i \cdot \CC^T\CC \ee_j}{\|\CC \ee_i\| \|\CC \ee_j\|} \nonumber\\ 
&=& \text{sgn}(x_k^*) \frac{ \ee_i \cdot \CC^T\CC \ee_k}{\|\CC \ee_i\| \|\CC \ee_k\|} \nonumber\\
&=& \text{sgn}(x_k^*) \frac{\CC \ee_i \cdot \CC \ee_k}{\|\CC \ee_i\| \|\CC \ee_k\|}. \label{eq:supeq}
\end{eqnarray}
Invoking Cauchy Schwartz' inequality, it follows that 
\begin{equation*}
    |\CC\ee_i \cdot \CC\ee_k| <  \|\CC \ee_i\| \|\CC \ee_k\|,
\end{equation*}
where the strict inequality can be asserted from the non-parallelism assumption \eqref{eq:nonpar}. Inserting this in \eqref{eq:supeq} gives
\begin{equation*}
 \left |\frac{\CC\ee_i}{\|\CC\ee_i\|}\cdot \cc \right|  < 1,
\end{equation*}
which shows that also condition \eqref{cond_2} of Theorem \ref{thm:support} is satisfied.

On the other hand, if $i \in \JJ^c$ and $i \notin supp(\CC^T\CC\ee_j)$ for any $j \in \JJ$, we get that
\begin{equation*}
    \frac{\CC \ee_i}{\|\CC \ee_i\|} \cdot \cc = 0,
\end{equation*}
showing that the condition \eqref{cond_2} also holds in this case. Thus, we can conclude that $\xx^*$ is a solution to the problem \eqref{II-Ax =Ax*}.

To prove the uniqueness, assume that there exists another minimizer $\yy$. Since both \eqref{cond_1} and \eqref{cond_2} are shown to hold, it follows from Theorem \ref{thm:support} that $\textnormal{supp}(\yy) \subset \textnormal{supp}(\xx^*)$.
Consequently, we can write $\AAA\xx_\alpha = \AAA\yy$ in the form
\begin{equation*}
    \AAA \sum_{j \in \mathcal{J}} y_j\ee_j =  \AAA \sum_{j \in \mathcal{J}} x_j^{*}\ee_j.
\end{equation*}
Furthermore, we can multiply with $\BB$ to obtain
\begin{equation*}
     \sum_{j \in \mathcal{J}}  y_j \CC \ee_j = \sum_{j \in \mathcal{J}} x_j^{*} \CC \ee_j
\end{equation*}
The orthogonality of $\{\CC \ee_j\}_{j \in \JJ}$ ensures that $y_j$ must be equal to $x_j^{*}$ for all $j \in \mathcal{J}$, which implies uniqueness.
 \end{proof}

\bibliographystyle{abbrv} %unsrt
\bibliography{references}

\end{document}